\documentclass[11pt,reqno,tbtags]{amsart}
\usepackage{amssymb}
\usepackage{url}
\usepackage[square,numbers]{natbib}
\bibpunct[, ]{[}{]}{;}{n}{,}{;}

\title{Euler--Frobenius numbers and rounding}

\date{15 May, 2013}

\author{Svante Janson}
\thanks{Partly supported by the Knut and Alice Wallenberg Foundation}
\address{Department of Mathematics, Uppsala University, PO Box 480,
SE-751~06 Uppsala, Sweden}
\email{svante.janson@math.uu.se}
\newcommand\urladdrx[1]{{\urladdr{\def~{{\tiny$\sim$}}#1}}}
\urladdrx{http://www2.math.uu.se/~svante/}

\subjclass[2010]{60C05; 05A15, 11B68, 41A15, 60E05, 60F05} 

\numberwithin{equation}{section}

\renewcommand\le{\leqslant}
\renewcommand\ge{\geqslant}

\allowdisplaybreaks

\newtheorem{theorem}{Theorem}[section]
\newtheorem{lemma}[theorem]{Lemma}

\newtheorem{corollary}[theorem]{Corollary}

\theoremstyle{definition}
\newtheorem{example}[theorem]{Example}

\newtheorem{remark}[theorem]{Remark}

\theoremstyle{remark}

\newenvironment{romenumerate}[1][0pt]{
\addtolength{\leftmargini}{#1}\begin{enumerate}
 }{\end{enumerate}}

\newcounter{oldenumi}
{\setcounter{oldenumi}{\value{enumi}}
\begin{romenumerate} \setcounter{enumi}{\value{oldenumi}}}
{\end{romenumerate}}

\newcounter{thmenumerate}

\newcounter{romxenumerate}   

\newcounter{xenumerate}   

\newcommand{\refT}[1]{Theorem~\ref{#1}}
\newcommand{\refC}[1]{Corollary~\ref{#1}}
\newcommand{\refL}[1]{Lemma~\ref{#1}}
\newcommand{\refR}[1]{Remark~\ref{#1}}
\newcommand{\refS}[1]{Section~\ref{#1}}

\newcommand{\refE}[1]{Example~\ref{#1}}

\newcommand{\refApp}[1]{Appendix~\ref{#1}}
\newcommand{\refTab}[1]{Table~\ref{#1}}
\newcommand{\refand}[2]{\ref{#1} and~\ref{#2}}

\newcommand\REM[1]{{\raggedright\texttt{[#1]}\par\marginal{XXX}}}

\begingroup
  \divide\count255 by 60
  \multiply\count255 by -60
  \advance\count255 by \time
  \ifnum \count255 < 10 \xdef\klockan{\the\count1.0\the\count255}
  \else\xdef\klockan{\the\count1.\the\count255}\fi
\endgroup

\newcommand\nopf{\qed}   

\newcommand{\sumj}{\sum_{j=0}^\infty}
\newcommand{\sumk}{\sum_{k=0}^\infty}
\newcommand{\sumkoooo}{\sum_{k=-\infty}^\infty}
\newcommand{\summ}{\sum_{m=0}^\infty}
\newcommand{\sumn}{\sum_{n=0}^\infty}

\newcommand{\sumin}{\sum_{i=0}^n}
\newcommand{\sumiin}{\sum_{i=1}^n}
\newcommand{\sumiini}{\sum_{i=1}^{n-1}}
\newcommand{\sumjin}{\sum_{j=1}^n}
\newcommand{\sumjn}{\sum_{j=0}^n}

\newcommand{\sumkn}{\sum_{k=0}^n}

\newcommand\set[1]{\ensuremath{\{#1\}}}

\newcommand\Bigset[1]{\ensuremath{\Bigl\{#1\Bigr\}}}

\newcommand\xpar[1]{(#1)}
\newcommand\bigpar[1]{\bigl(#1\bigr)}
\newcommand\Bigpar[1]{\Bigl(#1\Bigr)}

\newcommand\lrpar[1]{\left(#1\right)}

\newcommand\xcpar[1]{\{#1\}}

\def\rompar(#1){\textup(#1\textup)}    
\newcommand\xfrac[2]{#1/#2}

\newcommand\parfrac[2]{\lrpar{\frac{#1}{#2}}}

\newcommand\Bigparfrac[2]{\Bigpar{\frac{#1}{#2}}}

\def\xexp(#1){e^{#1}}
\newcommand\ceil[1]{\lceil#1\rceil}
\newcommand\floor[1]{\lfloor#1\rfloor}
\newcommand\Bigfloor[1]{\Big\lfloor#1\Big\rfloor}
\newcommand\bigfloor[1]{\big\lfloor#1\big\rfloor}
\newcommand\lrfloor[1]{\left\lfloor#1\right\rfloor}
\newcommand\frax[1]{\{#1\}}

\newcommand\ntoo{\ensuremath{{n\to\infty}}}

\newcommand\punkt[1]{\if.#1\else.\spacefactor1000\fi{#1}}
\newcommand\iid{i.i.d\punkt}    
\newcommand\ie{i.e\punkt}
\newcommand\eg{e.g\punkt}

\newcommand\cf{cf\punkt}

\newcommand{\aex}{a.e\punkt}

\newcommand\ii{\mathrm{i}}

\newcommand{\tend}{\longrightarrow}
\newcommand\dto{\overset{\mathrm{d}}{\tend}}
\newcommand\pto{\overset{\mathrm{p}}{\tend}}

\newcommand\eqd{\overset{\mathrm{d}}{=}}

\newcommand\bbR{\mathbb R}
\newcommand\bbC{\mathbb C}

\newcommand\bbZ{\mathbb Z}

\newcounter{CC}
\newcounter{cc}

\newcommand\E{\operatorname{\mathbb E{}}}
\renewcommand\P{\operatorname{\mathbb P{}}}
\newcommand\Var{\operatorname{Var}}

\newcommand\Be{\operatorname{Be}}
\newcommand\Ge{\operatorname{Ge}}

\newcommand\ga{\alpha}
\newcommand\gb{\beta}

\newcommand\gD{\Delta}
\newcommand\gf{\varphi}
\newcommand\gam{\gamma}

\newcommand\gl{\lambda}

\newcommand\go{\omega}

\newcommand\gss{\sigma^2}
\newcommand\eps{\varepsilon}

\renewcommand\phi{\xxx}  

\newcommand\cS{{\mathcal S}}

\newcommand\tU{{\widetilde U}}
\newcommand\tS{{\widetilde S}}

\newcommand\ett[1]{\boldsymbol1\xcpar{#1}}

\newcommand\etta{\boldsymbol1}

\newcommand\qw{^{-1}}

\newcommand\qq{^{1/2}}

\renewcommand{\=}{:=}

\newcommand\intoi{\int_0^1}
\newcommand\intoo{\int_0^\infty}
\newcommand\intoooo{\int_{-\infty}^\infty}
\newcommand\oi{[0,1]}

\newcommand\setoi{\set{0,1}}

\newcommand\dd{\,\mathrm{d}}
\newcommand\ddd{\mathrm{d}}
\newcommand\ddq[1]{\frac{\ddd}{\ddd #1}}

\newcommand{\pgf}{probability generating function}
\newcommand{\mgf}{moment generating function}
\newcommand{\chf}{characteristic function}

\newcommand\lhs{left-hand side}
\newcommand\rhs{right-hand side}

\newcommand\pnrho{P_{n,\rho}}
\newcommand\pnirho{P_{n-1,\rho}}
\newcommand\anrk{A_{n,k,\rho}}
\newcommand\EU{\mathfrak E}
\newcommand\EUrho{\EU_{n,\rho}}
\newcommand\EUx[1]{\EU_{n,#1}}
\newcommand\EUi{\EUx1}
\newcommand\EUo{\EUx0}
\newcommand\ZZ{Z}
\newcommand\znrho{\ZZ_{n,\rho}}

\newcommand{\euler}[2]{\genfrac{ < }{ > }{0pt}{}{#1}{#2}}
\newcommand\rr{$\rho$-rounding}
\newcommand\rrr[1]{\floor{#1}_\rho}
\newcommand\rrx[2]{\floor{#2}_{#1}}
\newcommand\xxn{(x_i)_1^n}
\newcommand\hP{\hat P}
\newcommand\eit{e^{\ii t}}
\newcommand\EFx{Euler--Frobe\-nius }
\newcommand\kk{\varkappa}

\newcommand\gdx[1]{\gD_{#1}}
\newcommand\gdr{\gdx{\rho}}
\newcommand\gdqq{\gdx{1/2}}
\newcommand\gdrgg{\gdx{\rho,\gam}}
\newcommand\hgd{\hat\gD}

\newcommand\xgD{\gD'}
\newcommand\ftd{f_{\hgd}}

\newcommand{\Polya}{P\'olya}

\begin{document}

\begin{abstract} 
We study the Euler--Frobenius numbers, a generalization of the Eulerian
numbers, and the probability distribution obtained by normalizing them.
This distribution can be obtained by rounding a sum of independent
uniform random variables; this is more or less implicit in various results
and we try to explain this and various connections to other areas of
mathematics, such as spline theory.

The mean, variance and (some) higher cumulants of the distribution are
calculated. 
Asymptotic results are given.
We include a couple of applications to rounding errors and election methods.
\end{abstract}

\maketitle

\section{Introduction}\label{S:intro}

The \emph{\EFx polynomial} $P_{n,\rho}(x)$ can be defined by
\begin{equation}
  \label{ef}
\frac{P_{n,\rho}(x)}{(1-x)^{n+1}}
=\Bigpar{\rho+x\frac{\ddd}{\dd x}}^n\frac1{1-x}
=
\sumj (j+\rho)^n x^j
,
\end{equation}
or, equivalently, by the recursion formula
\begin{equation}
  \label{efrec}
P_{n,\rho}(x)
=
\bigpar{nx+\rho(1-x)}P_{n-1,\rho}(x)+x(1-x)\pnirho'(x),
\qquad n\ge1,
\end{equation}
with $P_{0,\rho}(x)=1$;
see \refApp{AEF} for details, some further results and references.
Here $n=0,1,2,\dots$, and $\rho$ is a parameter that can be any complex
number, although we shall be interested mainly in the case $0\le\rho\le1$.
(The special cases $\rho=0,1$ yield the \emph{Eulerian polynomials}, see below.)

It is immediate from \eqref{efrec} that $\pnrho(x)$ is a polynomial in $x$
of degree at most $n$. We write
\begin{equation}
  \label{efa}
\pnrho(x)=\sumkn \anrk x^k.
\end{equation}
The recursion \eqref{efrec} can be translated to the recursion
\begin{equation}
  \label{efarec}
A_{n,k,\rho} = (k+\rho)  A_{n-1,k,\rho}  + (n-k+1-\rho)  A_{n-1,k-1,\rho},
\qquad n\ge1,
\end{equation}
where we let $A_{n,k,\rho}=0$ if $k\notin\set{0,\dots,n}$.
Following \cite{GawronskiN}, we call these numbers
\emph{\EFx numbers}.
See \refTab{tabrho} for the first values.
(The special cases $\rho=0,1$ yield the \emph{Eulerian numbers}, see below.)

We usually regard $\rho$ as a fixed parameter, but we note that $\pnrho(x)$
also is a polynomial in $\rho$, see \eqref{q21}; thus the
\EFx numbers $A_{n,k,\rho}$ are polynomials in $\rho$, as also
follows from \eqref{efarec}. 
(Some papers conversely consider $\pnrho(x)$ as a
polynomial in $\rho$, with $x$ as a parameter; see \eg{}
\cite{Carlitz,Schoenberg} and \refApp{Aspline}.)

\begin{table}
  \begin{tabular}{l|l|l|l|l}
	$n\backslash k$ & 0 & 1 & 2 & 3 
\\ \hline
0 & 1 & &&
\\ \hline
1 & $\rho$ & $1-\rho$ &
\\ \hline
2 & $\rho^2$ & $1+2\rho-2\rho^2$ & $1-2\rho+\rho^2$
\\ \hline
3 & $\rho^3$ & $1+3\rho+3\rho^2-3\rho^3$ & $4-6\rho^2+3\rho^3$
& $1-3\rho+3\rho^2-\rho^3$
\\ \hline
  \end{tabular}
\vskip 4pt 
\caption{The \EFx numbers $A_{n,k,\rho}$ for small $n$.}
\label{tabrho}
\end{table}

It follows from \eqref{efarec} that if $0\le\rho\le1$, then
$A_{n,k,\rho}\ge0$, so if we normalize by dividing by $\sumkn 
A_{n,k,\rho}=\pnrho(1)=n!$, see \eqref{pnrho1}, we obtain a probability
distribution on \set{0,\dots,n}; we call this distribution the
\emph{\EFx distribution} and let $\EU_{n,\rho}$ denote a random
variable with this distribution, \ie{}
\begin{equation}
  \label{efdist}
\P(\EU_{n,\rho}=k) = A_{n,k,\rho}/P_{n,\rho}(1)=A_{n,k,\rho}/n!.
\end{equation}
Equivalently,
$\EUrho$ has the \pgf{}
\begin{equation}\label{pgf}
  \E x^{\EUrho} 
=\sum_{k=0}^n\P(\EUrho=k)x^k
= \frac{\pnrho(x)}{\pnrho(1)}
=\frac{\pnrho(x)}{n!}.
\end{equation}
With a minor abuse of notation, we also denote this distribution by $\EUrho$.

\begin{remark}\label{R1}
Since $A_{1,0,\rho}=\rho$ and $A_{1,1,\rho}=1-\rho$, see \refTab{tabrho},
the condition
$0\le\rho\le1$ is also necessary for \eqref{efdist} to define a probability
distribution for all $n\ge1$.  
(We will extend the definition of $\EUrho$ to arbitrary $\rho$ later, see
\eqref{efreal}, but \eqref{efdist} holds only for $\rho\in\oi$.)
\end{remark}

The main purpose of the present paper is to show that this distribution
occurs when rounding sums of uniform random variables, and to give various
consequences and connections to other problems. Our main result can be
stated as follows.
(A proof  is given in \refS{SSn}.)

\begin{theorem}
  \label{T1}
Let $U_1,\dots,U_n$ be independent random variables uniformly distributed on
$\oi$, and let $S_n\=\sumiin U_i$. Then, for every $n\ge1$ and
$\rho\in\oi$, the random
variable $\floor{S_n+\rho}$ has the \EFx distribution
$\EUx{1-\rho}$, 
\ie,
\begin{equation}\label{t1}
  \P\bigpar{\floor{S_n+\rho}=k}=\frac{A_{n,k,1-\rho}}{n!},
\qquad k\in\bbZ.
\end{equation}
\end{theorem}

\refT{T1} can also be stated geometrically, see \refS{Svol}.

The distribution of $S_n$ was calculated already by Laplace
\cite[pp.~257--260]{Laplace}
(who used it for a statistical test showing that the orbits of the planets
are not randomly distributed, while the orbits of the comets seem to be
random
\cite[pp.~261--265]{Laplace}),
see also \eg{} \cite[Theorem I.9.1]{FellerII}.
The case $\rho=0$ (or $\rho=1$)
of \refT{T1}, which is a connection between the
distribution of $S_n$ and Eulerian numbers (see below), is well-known, see \eg{}
\cite{Tanny}, \cite{FoataNATO}, \cite{StanleyNATO}, \cite{Sachkov}. 
The case $\rho=1/2$, which means standard rounding of $S_n$,
is given in \cite{ChakerianL}.
Moreover, the theorem is implicit in
\eg{}
\cite[Lecture 3]{Schoenberg}, %
but I have not seen it stated explicitly in this form.
This paper is therefore partly expository, trying to explain some of the
many connections to other results in various areas.
(However, we do not attempt to give a complete history. Furthermore, there
are many papers on algebraic and other aspects of \EFx polynomials and
numbers that are not mentioned here.)

Before discussing rounding and \refT{T1} further, we return to the
\EFx polynomials and numbers to give some background.

The cases $\rho=0$ and $\rho=1$ are equivalent; we have 
$ P_{n,0}(x)=xP_{n,1}(x)$ and thus $A_{n,k+1,0}=A_{n,k,1}$ 
and $\EU_{n,0}\eqd\EU_{n,1}+1$
for all
$n\ge1$,
as follows directly from \eqref{ef} or by induction from \eqref{efrec} or
\eqref{efarec}. 
(Cf.\ \eqref{t1}, where the \lhs{} obviously has the corresponding property.)
This is the most important case and appears in many contexts.
(\citet{Carlitz} remarks that these polynomials and numbers have been
frequently rediscovered.) 
The numbers $A_{n,k,1}$ and the polynomials
$P_{n,1}$ were studied already by Euler \cite{E55,E212,E352}, and
the numbers 
$A_{n,k,1}$ are therefore called \emph{Eulerian numbers},
see \cite[A173018]{OEIS} and \refTab{tab1}; the usual modern
notation is
$\euler nk$ \cite{CM}, \cite[\S26.14]{NIST}.
Similarly, the polynomials $P_{n,1}(x)$ are usually called \emph{Eulerian
  polynomials}. 
(Notation varies, and these names are also used for the shifted versions
that we denote by
$A_{n,k,0}$ and $P_{n,0}(x)$, see \eg{} 
\cite[A008292, A173018 and A123125]{OEIS}.
Already Euler used both versions:
$P_{n,0}$ in \cite{E55} 
and $P_{n,1}$ in \cite{E212,E352}.)  

\begin{table}
  \begin{tabular}{l|r|r|r|r|r|r|r}
	$n\backslash k$ & 0 & 1 & 2 & 3 &4&5&6
\\ \hline
0 & 1 & &&&&&
\\ \hline
1 & $1$ &&&&&
\\ \hline
2 & $1$ & $1$ &&&&
\\ \hline
3 & 1 & 4 & 1 &&&
\\ \hline
4 & 1& 11& 11& 1 &&
\\ \hline
5 & 1& 26& 66& 26& 1 &
\\ \hline
6&  1& 57& 302& 302& 57& 1
\\ \hline
7 & 1& 120& 1191& 2416& 1191& 120& 1
\\ \hline
  \end{tabular}
\vskip 4pt 
\caption{The Eulerian numbers $A_{n,k,1}=A_{n,k+1,0}$ for small $n$.
The row sums are $n!$.}
\label{tab1}
\end{table}

Euler \cite{E55,E212,E352} used these numbers to calculate the sum of
series;
see also \cite{Hirzebruch} and \cite{Foata}.
(In particular, Euler \cite{E352} calculated the sum of the divergent series 
$\sum_{k=1}^\infty (-1)^{k-1} k^n$ for integers $n\ge0$; in modern terminology
he found the Abel sum as $2^{-n-1}P_{n,1}(-1)$ by letting $x\to-1$ in
\eqref{ef}.)
They have since appeared in many other contexts.
For example,
the Eulerian number $A_{n,k,1}=\euler nk$ equals
the number of permutations of length $n$ with $k$ descents (or ascents),
see \eg{} 
\cite[Chapter 8.6]{Riordan}, 
\cite[Chapter 10]{DavidB} or \cite[Section 1.3]{StanleyI};
this well-known combinatorial interpretation is often taken as the
definition of Eulerian numbers.
(In the terminology introduced above, the number of descents
in a  random permutation thus has the \EFx distribution $\EUi$.)
Furthermore, the Eulerian numbers also enumerate permutations 
with $k$ exceedances, see again \cite[Section 1.3]{StanleyI}, where also
further related combinatorial interpretations are given.
See also \cite{cd-h} for an enumeration with staircase tableaux, and
\cite{FoataSch} for further related results.
Some other examples where Eulerian numbers and polynomials appear are
number theory \cite{Frobenius}, \cite[p.~328]{MacMahon}, 
summability \cite[p.~99]{Peyerimhoff},
statistics \cite{Kimber},
control theory \cite{Weller+},
and splines
\cite{Schoenberg,Schoenberg-LNM}, 
\cite[Table 2, p.~137]{Schumaker}
(see also \refApp{Aspline}).

The case $\rho=1/2$ also occurs in several contexts.
In this case, it is often more convenient to consider the numbers 
$B_{n,k}\=2^n  A_{n,k,1/2}$ 
which are integers and satisfy the recursion
\begin{equation}
  \label{efbrec}
B_{n,k} = (2k+1)  B_{n-1,k}  + (2n-2k+1)  B_{n-1,k-1},
\qquad n\ge1,
\end{equation}
with $B_{0,0}=1$ (and $B_{n,k}=0$ if $k\notin\set{0,\dots,n}$),
see \cite[A060187]{OEIS} and \refTab{tab1/2}.
These numbers are sometimes
called \emph{Eulerian numbers of type B}.
They seem to have been introduced by \citet[p.~331]{MacMahon} 
in number theory.
(It seems likely that they were used already by Euler, who in \cite{E352}
also says, without giving the calculation, that he can prove similar results
for $\sum_{k=1}^\infty (-1)^{k} (2k-1)^m$;
see \cite{Hirzebruch} for a calculation using 
$P_{n,1/2}$ and methods of \cite{E352}.)
The numbers $B_{n,k}$ also have combinatorial interpretations, for example as
the numbers of descents in signed permutations, \ie, in the
hyperoctahedral group
\cite{Brenti,ChowG,SchmidtS}.
Furthermore,
the numbers $B_{n,k}$ and the distribution $\EU_{n,1/2}$ appear  in the
study of random staircase tableaux \cite{d-hh}.
$P_{n,1/2}(x)$ and $B_{n,k}$ appear in spline theory
\cite[Lecture 3.4]{Schoenberg} (see \refApp{Aspline}).
They also appear 
(as do $P_{n,1}(x)$ and $A_{n,k,1}$) in \cite{Franssens}, as special cases of
more general polynomials. 
\begin{table}
  \begin{tabular}{l|r|r|r|r|r|r|r}
	$n\backslash k$ & 0 & 1 & 2 & 3 &4&5&6
\\ \hline
0 & 1 & &&&&&
\\ \hline
1 & $1$ & 1& &&&&
\\ \hline
2 & 1& 6& 1& &&&
\\ \hline
3 & 1& 23& 23& 1& &&
\\ \hline
4 & 1& 76& 230& 76& 1& &
\\ \hline
5 & 1& 237& 1682& 1682& 237& 1&
\\ \hline
6&  1& 722& 10543& 23548& 10543& 722& 1
\\ \hline
  \end{tabular}
\vskip 4pt 
\caption{The Eulerian numbers of type B, 
$B_{n,k}=2^nA_{n,k,\frac12}$, for small $n$.
The row sums are $2^nn!$.}
\label{tab1/2}
\end{table}

The general polynomials $\pnrho$ 
were perhaps first introduced by \citet{Carlitz} 
(in the form $\pnrho(x)/(x-1)^n$, \cf{} \eqref{carl} below). 
They  are important in spline theory, see \eg{}
\cite[Lecture 3]{Schoenberg}, \cite{Schoenberg-LNM},
\cite{terMorsche}, \cite{ReimerS} and \refApp{Aspline}.
They  appear also (as a special case) in
the study of random staircase tableaux  \cite{SJ276}.
Note also that
the function \eqref{ef} is the special case $s\in\set{0,-1,-2,\dots}$ of
Lerch's transcendental function $\Phi(z,s,\rho)=\sumj (j+\rho)^{-s} z^j$,
see  \cite{Lerch}, \cite[\S25.14]{NIST} and 
\eg{} \cite{Truesdell} with further references.
The  general Eulerian Numbers $A_{n,k}(a,d)$ defined by
\cite{Xiong-etal} equal our $d^nA_{n,k+1,1-a/d}$.

The special case $\rho=1/N$ where $N\ge1$ is an integer appears in
combinatorics. The integers $N^n A_{n,k,1/N}$ 
(\cf{} $B_{n,k}$ above, which is the case $N=2$)
enumerate indexed permutations with $k$ descents (or with $k$
exceedances), 
generalizing the cases $N=1$ (permutations) and
$N=2$ (signed permutations) above, see \cite{Steingrimsson}.

\begin{remark}
Frobenius \cite{Frobenius} studied the Eulerian polynomials $P_{n,1}$
in detail (with applications to number theory); he also gave them the name
Eulerian (in German).
The Eulerian polynomials have sometimes been called 
\emph{Euler--Frobenius polynomials} (see \eg{} \cite[p.~22]{{Schoenberg}}
and \cite{Weller+}), and 
the generalization \eqref{ef} considered here has been called
\emph{generalized Euler--Frobenius polynomials}
by various authors (\eg{} \cite{Merz2,Reimer:extremal,Siepmann,ReimerS}),
but this has also been simplified by dropping ``generalized''  
and calling them too just
\emph{Euler--Frobenius polynomials} 
(\eg{} \cite{Merz1,Reimer:main,GawronskiS,GawronskiN}). 
We follow the latter usage, for convenience rather than for historical accuracy.
(As far as I know, neither Euler nor Frobenius considered this generalization.)
The names \emph{Frobenius--Euler polynomials} and \emph{numbers} are also used
in the literature (\eg{} \cite{Simsek-etal}).
The reader should note that also other generalizations of Eulerian
polynomials have been called Euler--Frobenius polynomials,
and that, conversely, other names have been used for our \EFx
polynomials \eqref{ef}.
Note futher that \emph{Euler numbers} and \emph{Euler polynomials} (usually)
mean
something different, see \refR{Reuler}.
\end{remark}

\begin{remark}
As said above, the notation varies.
Examples of other notations for our $\pnrho(x)$ are
$H_n(\rho,x)$ (\eg{} \cite{Merz1,ReimerS}) 
and $P_n(x,1-\rho)$ (\eg{} \cite{terMorsche:relations}).
A different notation used by \eg{} \citet{Frobenius} and \citet{Carlitz}
(in the classical case $\rho=1$) is $R_n = P_{n,1}$ and 
$H^n$ or
$H_n(x)=P_{n,1}(x)/(x-1)^n$. 
\citet{Carlitz} uses for the general case
\begin{equation}\label{carl}
H_n(u\mid\gl)=\frac{P_{n,1-u}(\gl)}{(\gl-1)^n}=\sumjn \binom nj u^{n-j}H_j(\gl),
\end{equation}
where the last equality follows from \eqref{q211}.
Similarly, \eg{} \citet{Schoenberg} uses $A_n(x;t)$ for our 
$(1-t\qw)^{-n}P_{n,x}(t\qw)=(t-1)^{-n}P_{n,1-x}(t)$ (which thus equals
$H_n(x\mid t)$ in \eqref{carl}), \cf{} \eqref{psymm}; he further uses
$\Pi_n(t)$ for $P_{n,1}(t)$ and $\rho_n(t)$ for $2^nP_{n,1/2}(t)$.
\end{remark}

\begin{remark}
Many other combinatorial numbers satisfy 
recursion formulas similar to \eqref{efarec}; see
\cite{WangYeh} for a general version.
There are also many other generalizations of Eulerian numbers and
polynomials that have been defined by various authors; for a few examples,
see
\cite{Carlitzq,Carlitzq2},
\cite{Carlitz-higher},
\cite{Carlitz-Scoville},
\cite{Dumont},
\cite{Tsumura}, 
\cite{Simsek-etal},
\cite{Simsek},
\cite{Xiong-etal}.
In particular, note the generalized Eulerian numbers
$A(r,s\,|\, \ga,\gb)$ defined by 
\citet{Carlitz-Scoville}; the \EFx numbers are the special case
$A_{n,k,\rho} = A(n-k,k\,|\,\rho,1-\rho)$.
\end{remark}

As said above, the case $\rho=0$ (or $\rho=1$) of \refT{T1} is well-known.
We end this section by recalling the simple
proof by Stanley \cite{StanleyNATO}
giving an explicit connection between $\floor{S_n}$ and  the number of
descents in a 
random permutation,
which, as said above, has the distribution $\EUi$;
we give it here in probabilistic formulation rather than the original geometric,
\cf{} \refT{T2}: 

In the notation of \refT{T1}, let $V_j$ be the fractional part
$\frax{S_j}\=S_j-\floor{S_j}$; 
then $V_1,\dots,V_n$ is another sequence of independent uniformly
distributed random variables. Thus the number of descents in a random
permutation of length $n$ 
has the same distribution as 
$\sum_{i=2}^{n} \ett{V_{i-1}>V_{i}}$. On the other hand, $V_{i-1}>V_{i}$
exactly when the sequence $S_1,\dots,S_n$ passes an integer;
thus $\ett{V_{i-1}>V_{i}}=\floor{S_i}-\floor{S_{i-1}}$ and this
sum equals $\floor{S_n}$. 

An extension of this proof to the case
$\rho=1/N$ and indexed permutations is given in 
\cite[Theorem 50]{Steingrimsson};
a modification for the case $\rho=1/2$ and 
(one version of) descents in signed
permutations is given in \cite{SchmidtS}.

\refS{Svol} gives a geometric formulation of \refT{T1} and some related
results.
\refS{SSn} gives a proof of \refT{T1} together with further connections
between the distribution of $S_n$ and \EFx  numbers.
\refS{Srrr} introduces \rr, and states \refT{T1} using it.
\refS{Schf} uses this to derive results on the characteristic function and
moments of the \EFx distribution.
\refS{SCLT} shows asymptotic normality and gives further asymptotic results.
\refS{Sval} gives applications to a well known problem on rounding.
\refS{Share} gives  applications to an election method. 
Finally, the appendices give further background and connections to other
results. 

We let throughout $U$ and $U_1,U_2,\dots$ denote independent uniform random
variables in $\oi$, and $S_n\=\sumiin U_i$.

\section{Volumes of slices}\label{Svol}

\refT{T1} can also be stated
geometrically as follows. A proof is given in \refS{SSn}.
\begin{theorem}
  \label{T2}
Let $Q^n\=\oi^n$ be the $n$-dimensional unit cube and let, for $s\in\bbR$,
$Q^n_s$ be the slice
\begin{equation}\label{t2}
  Q^n_s\=\Bigset{(x_i)_1^n\in Q^n:s-1\le \sumiin x_i\le s}.
\end{equation}
Then the volume of $Q^n_{k+\rho}$
is
$\xfrac{A_{n,k,\rho}}{n!}$, 
for all $n\ge1$, $k\in\bbZ$ and $\rho\in\oi$.
\end{theorem}

As said above, the case $\rho=0$ (or $\rho=1$), when the volumes are given
by the Eulerian numbers $A_{n,k,1}$, is well-known
\cite{FoataNATO,StanleyNATO, Hensley,ChakerianL, SchmidtS}.

The case $\rho=1/2$, which corresponds to standard rounding of $S_n$ in
\refT{T1}, and where the result can be stated using the Eulerian numers of
type B $B_{n,k}$ in \eqref{efbrec}, is treated in \cite{ChakerianL}
(with reference to an unpublished technical memorandum \cite{Slepian})
and in \cite{SchmidtS}.

\cite{ChakerianL} gives also the $(n-1)$-dimensional area of the 
slice \set{\xxn\in Q^n:\sumiin x_i=s}.
This equals, by simple geometry, $\sqrt n$ times the density function of
$S_n$ at $s$,  which by \eqref{t3x} below equals
(except when $n=1$ and $s=1$)
\begin{equation}
 \frac{ \sqrt{n}}{(n-1)!} A_{n-1,\floor{s},\frax s}.
\end{equation}

See also \cite{Polya} and \cite{ChakerianL}, and the further references in
the latter, for related results on 
more general slices of cubes. Furthermore, 
\cite{SchmidtS} give related results, involving Eulerian numbers, on some
slices of a simplex.

Mixed volumes of two consequtive slices $Q^n_k$ and $Q^n_{k+1}$ (with
integer $k$) are studied by \cite{ERS}, and further by \cite{WangXuXu}
where relations to our $A_{n,k,\rho}$ (and $f_{n+1}(x)$ in \refT{T3} below)
are given
based on the fact \cite{ERS} that the Minkowski sum 
$\gl Q^n_k+Q^n_{k+1}=(\gl+1) Q^n_{k+1/(\gl+1)}$.

\section{The distribution of $S_n$}\label{SSn}
Let $F_n(x)$ be the distribution function and $f_n(x)=F_n'(x)$ the density
function of $S_n\=\sumiin U_i$. 
Then $f_1(x)$, the density function of $S_1=U_1$, is the
indicator function $\etta_{\oi}$ of the interval $\oi$, and $f_n$ is the
$n$-fold convolution $\etta_{\oi}*\dotsm*\etta_{\oi}$.
Hence, for $n\ge1$,
\begin{equation}
  \label{sw}
f_{n+1}(x)=f_n*\etta_{\oi}(x) = \intoi f_n(x-y)\dd y
=F_n(x)-F_n(x-1).
\end{equation}
Note that the density $f_n(x)$
is continuous for $n\ge2$,
as a convolution of bounded, integrable functions
(or by \eqref{sw}, since $F_n(x)$ is continuous for $n\ge1$), 
while $f_1(x)$ is discontinuous at $x=0$ and $x=1$. We regard $f_1(x)$ as
undetermined at these two points, and we will tacitly assume that $x\neq0,1$ 
in equations involving $f_1(x)$ (such as \eqref{fn} when $n=1$).

The distribution of $S_n$ was, as said above,  
calculated already by \citet[pp.~257--260]{Laplace}
(by taking the limit of a discrete version), 
see also \eg{} 
\citet[XI.7.20]{FellerI} (where the result is attributed to Lagrange)
and 
\cite[Theorem I.9.1]{FellerII} (with a simple proof using \eqref{sw} and
induction); the result is the following.
(A more general formula for the sum of independent
uniform random variables on different intervals is given by \citet{Polya}.)
We use the notation $(x)_+^n\=(\max(x,0))^n$, interpreted as $0$ when $x\le0$
and $n\ge0$.
\begin{theorem}[E.g.\ \cite{Laplace}, \cite{FellerII}]
For $n\ge1$,
$S_n$ has the distribution function
\begin{equation}\label{Fn}
  F_n(x)\=\P(S_n\le x) = \frac1{n!}\sumjn (-1)^j \binom nj (x-j)_+^n
\end{equation}
and density function
\begin{equation}\label{fn}
f_n(x)\=F'_n(x) = \frac1{(n-1)!}\sumjn (-1)^j \binom nj (x-j)_+^{n-1}.
\end{equation}
\nopf 
\end{theorem}

It is easy to see that Theorems \refand{T1}{T2} are equivalent to the
following relation between the densities $f_n$ and the \EFx
numbers.
(
The case $\rho=0$ is noted in \cite{Esseen}. Moreover,
the relation is well-known in the spline setting, see \eg{} 
\cite[Theorem 3.2]{Schoenberg} (for $\rho=1$) and 
\cite{terMorsche,terMorsche:relations,Siepmann-Dr,Siepmann};
see also (for the case $\rho=0$ or $1$) \cite{WangXuXu,He}.)
\begin{theorem}\label{T3}
  For integers $n\ge0$ and $k\in\bbZ$, and $\rho\in\oi$,
  \begin{equation}\label{t3}
	f_{n+1}(k+\rho)=\frac{A_{n,k,\rho}}{n!}.
  \end{equation}
Equivalently, for every real $x$,
\begin{equation}\label{t3x}
  f_{n+1}(x)=\xfrac{A_{n,\floor{x},\frax x}}{n!}.
\end{equation}
\end{theorem}

We first verify that, as claimed above, Theorems \ref{T1}, \ref{T2} and
\ref{T3} are 
equivalent; we then prove the three theorems.

\begin{proof}[Proof of \refT{T1}$\iff$\refT{T2}$\iff$\refT{T3}]
By replacing $\rho$ by $1-\rho$, \eqref{t1} can be written
\begin{equation}\label{t1-}
  \P\bigpar{\floor{S_n+1-\rho}=k}=\frac{A_{n,k,\rho}}{n!},
\qquad \rho\in\oi,\;k\in\bbZ.
\end{equation}

In \refT{T2}, 
the volume of $Q^n_{s}$ equals 
$\P\bigpar{s-1\le S_n < s}$. Taking $s=k+\rho$, we have
\begin{equation*}
  \begin{split}
k+\rho-1\le S_n< k+\rho
&\iff k\le S_n+1-\rho< k+1
\\&
\iff \floor{S_n+1-\rho}=k,	
  \end{split}
\end{equation*}
and thus \refT{T2} is equivalent to \eqref{t1-}.

Similarly, 
  by \eqref{sw},
at least when $n\ge1$, 
  \begin{equation}\label{swa}
f_{n+1}(k+\rho)=\P(k+\rho-1< S_n \le k+\rho) 	
=\P(\floor{S_n+1-\rho}=k),
  \end{equation}
so \eqref{t3} is equivalent to \eqref{t1-}.
Hence all three theorem are equivalent to \eqref{t1-}.
(The trivial and partly exceptional case $n=0$ of \refT{T3} can be
verified directly.) 
\end{proof}

\begin{proof}[Proof of Theorems \ref{T1}, \ref{T2}, \ref{T3}]
It suffices to prove one of the theorems; we chose the version \eqref{t3x}.
We do this by calculating the Laplace transform of both sides, showing that
they are equal. (Note that both sides vanish for $x<0$.)
This implies that the two sides are equal \aex, and since both sides are
continuous on each interval $[k,k+1)$, they are equal for every $x$.
(In the trivial case $n=0$, we exclude $x=0,1$ as said above.)
Alternatively, for $n\ge1$, we can see directly that both sides of
\eqref{t3x} are continuous on $\bbR$,
using \eqref{a01} for the \rhs.

The Laplace transform of $f_{n+1}(x)$ is
\begin{equation}
  \begin{split}
\intoo f_{n+1}(x)e^{-sx}\dd x
&=\E e^{-s S_{n+1}}
=\bigpar{\E e^{-s U_1}}^{n+1}
=\lrpar{\intoi e^{-sx}\dd x}^{n+1}
\\&
=\parfrac{1-e^{-s}}{s}^{n+1}.
  \end{split}
\end{equation}

For $A_{n,\floor{x},\frax x}$ we obtain, using \eqref{efa} and
\eqref{ef},
\begin{equation}
  \begin{split}
\intoo & A_{n,\floor{x},\frax x}  e^{-sx}\dd x
=\sumk \intoi A_{n,k,\rho}e^{-s(k+\rho)}\dd\rho
\\&
= \intoi e^{-s\rho}\sumk A_{n,k,\rho}e^{-sk}\dd\rho
= \intoi e^{-s\rho} P_{n,\rho}(e^{-s})\dd\rho
\\&
= \intoi e^{-s\rho} \bigpar{1-e^{-s}}^{n+1}\sumj (j+\rho)^ne^{-js} \dd\rho
\\&
=  \bigpar{1-e^{-s}}^{n+1}\intoo x^ne^{-sx} \dd x
= n! \parfrac{1-e^{-s}}{s}^{n+1}.
  \end{split}
\end{equation}
hence the two Laplace transforms are equal, which completes the proof.
\end{proof}

An alternative  proof is by induction in $n$, using
the derivative of \eqref{sw} and \eqref{kem}. We leave this to the reader.

\begin{remark}
By \eqref{t3}, the basic recursion \eqref{efarec} is equivalent to
the recursion formula
\begin{equation}\label{frec}
  f_{n+1}(x)=\frac1n\bigpar{xf_n(x)+(n+1-x)f_n(x-1)}
\end{equation}
for the density functions $f_n$.
This formula is well-known in spline theory, see
 \cite[(4.52)--(4.53)]{Schumaker}. 
Conversely, \eqref{frec} implies \eqref{t3}, and thus also Theorems
\ref{T1} and \ref{T2}, by induction.
\end{remark}

\begin{remark}\label{Rfun}
  By \eqref{t3} and \eqref{fn}, for $n\ge1$ and $\rho\in\oi$,
\begin{equation}\label{fun}
A_{n,k,\rho}
= \sum_{j=0}^{n+1} (-1)^j \binom {n+1}j (k+\rho-j)_+^{n}
= \sum_{j=0}^{k} (-1)^j \binom {n+1}j (k+\rho-j)^{n}.
\end{equation}
This is another well-known formula, at least for the Eulerian case $\rho=1$.
It has been used to extend the definition of the \EFx numbers
to arbitrary real $n$ by \cite{ButzerH} (Eulerian numbers, $\rho=1$)
and \cite{LesieurN}
(note that $A(x,n)$ in \cite{LesieurN} equals our $A_{n,\floor{x},\frax{x}}$).
\end{remark}

We extend the definition \eqref{efdist} of $\EUrho$ for $\rho\in\oi$ to
arbitrary real $\rho$ by defining, for any $\rho\in\bbR$,
\begin{equation}
  \label{efreal}
\EUrho\=\EU_{n,\frax\rho}-\floor\rho.
\end{equation}
As above, we use $\EUrho$ also to denote the distribution of this random
variable. 
Since $\EU_{n,0}\eqd\EU_{n,1}+1$, and we only are interested in the
distribution of $\EUrho$, \eqref{efreal} is consistent with our previous
definition 
\eqref{efdist} for all $\rho\in\oi$. Note,  however, that \eqref{efdist}
holds only for $\rho\in\oi$. (See also \refR{R1}.)

This rather trivial extension is sometimes convenient. Theorems \ref{T1} and
\ref{T3} extend immediately:

\begin{theorem}
  \label{T4}
For any real $\rho$ and $n\ge1$,
\begin{equation}
  \EUrho\eqd\floor{S_n+1-\rho}
\end{equation}
and, for any $k\in\bbZ$,
\begin{equation}\label{t4b}
  \P(\EUrho=k)=\P(\floor{S_n+1-\rho}=k)
=A_{n,k+\floor \rho,\frax \rho}/n! = f_{n+1}(k+\rho)
\end{equation}
and
\begin{equation}\label{t4c}
  \P(\EUrho\le k)=\P(\floor{S_n+1-\rho}\le k)
=F_n( k+\rho).
\end{equation}
\end{theorem}

\begin{proof}
  By \eqref{efreal} and \refT{T1}, 
  \begin{equation*}
	\EUrho
\=
\EU_{n,\frax\rho}-\floor{\rho}\eqd\floor{S_n+1-\frax\rho}-\floor\rho
=\floor{S_n+1-\rho},
  \end{equation*}
which also implies \eqref{t4c}.
Moreover, by \eqref{efreal}, \eqref{efdist} and \eqref{t3x},
\begin{equation*}
  \P(\EUrho=k)=\P\bigpar{\EU_{n,\frax\rho}=k+\floor\rho}
=A_{n,k+\floor \rho,\frax \rho}/n! = f_{n+1}(k+\rho).
\qedhere
  \end{equation*}
\end{proof}

\section{Rounding}\label{Srrr}

Let $\rho\in\oi$ and define \emph{\rr} of real numbers by rounding a number $x$
down (to the nearest integer) if its fractional part $\frax x$ is less that
$\rho$, and up (to the nearest integer) otherwise.
We denote \rr{} by $\rrr{x}$, and can state the definition as
\begin{equation}
  \rrr x=n \iff n-1+\rho \le x <n+\rho,
\end{equation}
or, equivalently,
\begin{equation}\label{rrr}
  \rrr x = \floor{x+1-\rho}.
\end{equation}

In particular, $\floor{x}_1=\floor{x}$ (rounding down),
$\floor{x}_0=\ceil{x}$ (rounding up), except when $x$ is an integer,
and $\rrx{1/2} x$ is standard rounding (except perhaps when $\frax{x}=1/2$).

\begin{remark}
  As seen from these examples, in the case $\frax x=\rho$, the definition
  made above (for definiteness) is not obviously the best choice. Often
it is better to leave this case ambiguous, allowing rounding both up and
down.
However, we will  be interested in roundings of continuous random variables,
and then this exceptional case has  probability 0 and may be ignored.
\end{remark}

We define $\rrr x$ by \eqref{rrr} for arbitrary $\rho\in\bbR$.
This will be convenient later, although it is strictly speaking not  a
``rounding'' 
when $\rho\notin\oi$.

\begin{example}\label{Eval}
  One use of \rr{} is in the study of election methods; more precisely
  methods for proportional elections using party lists.
(In the United States, such methods are used, under different names,
for apportionment of the seats in  the House of Representatives among the
states.) 
Several important such methods are \emph{divisor methods}, and most of them 
can be
described as giving a party with $v$ votes $\rrr{v/D}$ seats, where $\rho$
is a given number and the divisor $D$ is chosen such that the total number
of seats is a predetermined number (the house size). The main examples are
$\rho=0$ (\emph{d'Hondt's method} = \emph{Jefferson's method})
and 
$\rho=1/2$ (\emph{Sainte-Lagu\"e's method} = \emph{Webster's method}).
Some other important  proportional election methods  are 
\emph{quota methods}, which again 
can be described as giving a party with $v$ votes $\rrr{v/D}$ seats, 
where now $D$ (in this setting called the \emph{quota})
is given by some formula and $\rho$ is
chosen such that the total number of seats is the house size. 
The most important example is to take $D$ as the \emph{simple quota} (also
called \emph{Hare quota}), \ie, the average number of votes per
seat 
(\emph{the method of greatest remainder} = \emph{Hare's method} =
\emph{Hamilton's method}). 
We return to election methods in Sections \refand{Sval}{Share}.
See further \cite[Appendices A and B]{SJ262} and \cite{BY}, 
\cite{Kopfermann},
\cite{Pukelsheim}.
(In the study of election methods, usually $\rho\in\oi$, but occasionally
other values of $\rho$ are used, see \cite{SJ262}.)
\end{example}

By \eqref{rrr}, yet another formulation of \refT{T1} is the following.
\begin{theorem}\label{Trr}
For every $\rho\in\bbR$ and $n\ge1$, the random
variable $\rrr{S_n}$ has the \EFx distribution
$\EUx{\rho}$.
In particular, if 
$\rho\in\oi$, then
\begin{equation}\label{trr}
  \P\bigpar{\floor{S_n}_\rho=k}
=\P(\EUrho=k)
=\frac{A_{n,k,\rho}}{n!},
\qquad k\in\bbZ,
\end{equation}
and more generally, for any real $\rho$,
\begin{equation}\label{trrr}
  \P\bigpar{\floor{S_n}_\rho=k}
=\P(\EUrho=k)
=\frac{A_{n,k+\floor\rho,\frax\rho}}{n!},
\qquad k\in\bbZ.
\end{equation}
\end{theorem}
\begin{proof}
The first claim is
 immediate from \eqref{rrr} and \refT{T4}.
This yields \eqref{trr} by \eqref{efdist} and then \eqref{trrr} by
\eqref{efreal}. 
\end{proof}

In other words, defining $\znrho\=\rrr{S_n}$, we have
\begin{equation}\label{znrho}
\znrho\=\rrr{S_n}=\floor{S_n+1-\rho}\sim\EUrho.
\end{equation}
In particular, when $\rho\in\oi$, $\znrho$ has the \pgf{} \eqref{pgf}.

\citet{SJ175} studied roundings using the notation, for $\ga\in\bbR$,
\begin{equation}
 X_\ga\=\floor{X+\ga}-\ga+1. 
\end{equation}
Comparing with \eqref{znrho}, we see that in this notation,
\begin{equation}\label{znrho175}
  \znrho=(S_n)_{1-\rho}-\rho.
\end{equation}

We state a corollary of \refT{Trr}
for standard rounding (\ie, $\floor{x}_{1/2}$), which
again shows 
the special importance of the cases $\rho=0,\frac12,1$ of the \EFx
numbers. 

\begin{corollary}
  \label{C1}
Let $\tU_1,\dots,\tU_n$ be independent random variables uniformly distributed on
$[-\frac12,\frac12]$, and let $\tS_n\=\sumiin \tU_i$, with
$n\ge1$.
Then,
$\floor{\tS_n}_{1/2}$ has the distribution
$\EU_{n,(n+1)/2}$, and thus, for $k\in\bbZ$,
\begin{equation*}
  \P\bigpar{\floor{\tS_n}_{1/2}=k}=
  \begin{cases}
\xfrac{A_{n,k+n/2,1/2}}{n!},	& n \text{ even},
\\
\xfrac{A_{n,k+(n+1)/2,0}}{n!}=
\xfrac{A_{n,k+(n-1)/2,1}}{n!},	& n \text{ odd}.
  \end{cases}
\end{equation*}
\end{corollary}
\begin{proof}
  We can take $\tU_i\=U_i-\frac12$, and then, using \eqref{znrho},
  \begin{equation*}
	\floor{\tS_n}_{1/2} = \Bigfloor{\tS_n+\frac12}
=\Bigfloor{S_n-\frac{n-1}2}
=\bigfloor{S_n}_{(n+1)/2}\eqd \EU_{n,(n+1)/2}.
  \end{equation*}
The result follows by \eqref{trrr}.
\end{proof}

\section{Characteristic function and moments}\label{Schf}

We use the results in \refS{Srrr} to derive further results for
the \EFx distribution $\EUrho$. 
As said in the introduction, we 
also use $\EUrho$ to denote a random variable with this distribution.
Since $\EUrho\eqd\znrho$ by \eqref{znrho}, we can just as well consider
$\znrho\=\rrr{S_n}$.

We begin with an expression for the \chf{} and \mgf{} of 
the \EFx distribution $\EUrho$.
(Cf.\ \cite[Lemma 2.4]{GawronskiN} where an equivalent formula is given.)
We denote the \chf{} of a random variable $X$ by $\gf_X$.
Note that if $\rho\in\oi$ 
(and the general case can be reduced to this by \eqref{efreal}), 
we have by \eqref{pgf}
\begin{equation}\label{efchf0}
	\gf_{\EUrho}(t)
\= \E e^{\ii t\EUrho}
=\frac{\pnrho(\eit)}{n!}
  \end{equation}
and, more generally,
for all $t\in\bbC$,
the \mgf{} 
  \begin{equation}\label{efmgf0}
\E e^{t\EUrho}
=\frac{\pnrho(e^t)}{n!}.
  \end{equation}

\begin{theorem}
Let $n\ge1$ and $\rho\in\bbR$.
  The \chf{} of $\EUrho$ is given by
  \begin{equation}\label{efchf}
	\gf_{\EUrho}(t)
=\ii^{-n-1}e^{-\ii\rho t}\bigpar{\eit-1}^{n+1}\sum_{k=-\infty}^\infty 
\frac{ e^{-2\pi\ii k\rho}}{(t+2\pi k)^{n+1}}.
  \end{equation}
Equivalently, the \mgf{} is, for all $t\in\bbC$,
  \begin{equation}\label{efmgf}
\E e^{t\EUrho}
=e^{-\rho t}\bigpar{e^t-1}^{n+1}\sum_{k=-\infty}^\infty 
\frac{ e^{-2\pi\ii k\rho}}{(t+2\pi k\ii)^{n+1}}.
  \end{equation}
(For $t\in 2\pi\bbZ$ or $t\in2\pi\ii\bbZ$, respectively,
the expressions are interpreted by continuity.) 
\end{theorem}

\begin{proof}
This can be reduced to the case
$\rho\in\oi$, and then \eqref{efchf} is \eqref{efchf0}
together with a special case  of an expansion found by 
\citet[(4)  and (5)]{Lerch} (take $s=-n$ there); 
nevertheless, we give a probabilistic proof (valid for all $\rho$)
using a general formula for rounded stochastic variables  in 
\cite{SJ175}.

Since $U_i$ has the \chf{} 
\begin{equation}
  \gf_{U}(t)\=\E e^{\ii tU}=\frac{\eit-1}{\ii t},
\end{equation}
the sum $S_n$ has the \chf{}
\begin{equation}\label{chfsn}
  \gf_{S_n}(t)=\gf_U(t)^n = \Bigparfrac{\eit-1}{\ii t}^n.
\end{equation}
The formula in \cite[Theorem 2.1]{SJ175} now yields
\begin{equation*}
  \begin{split}
  \E e^{\ii t (S_n)_{1-\rho}}
&=\sum_{k=-\infty}^\infty 
 e^{2\pi\ii k(1-\rho)}\gf_U(t+2\pi k)\gf_{S_n}(t+2\pi k)
\\&
=\sum_{k=-\infty}^\infty 
 e^{2\pi\ii k(1-\rho)}\parfrac{\eit-1}{\ii(t+2\pi k)}^{n+1},	
  \end{split}
\end{equation*}
which yields \eqref{efchf} by \eqref{znrho} and \eqref{znrho175}.

This derivation tacitly assumes that $t$ is real, but the sum in
\eqref{efchf} converges for every complex $t\in\bbC\setminus2\pi\bbZ$, and
defines a meromorphic function with poles in $2\pi\bbZ$; thus the \rhs{} of
\eqref{efchf} is an entire function of $t\in\bbC$. So is also the \lhs,
since it is a trigonometric polynomial. Hence \eqref{efchf} is valid for all
complex $t$, and replacing $t$ by $t/\ii$, we obtain \eqref{efmgf}.
\end{proof}

Moments of arbitrary order can be obtained from the \mgf{} by
differentiation. For moments of order at most $n$, this leads to simple
results. 

\begin{lemma}
  \label{Lmom}
Let $n\ge1$ and $\rho\in\bbR$.
The random variables $\EUrho+\rho$ and $S_{n+1}$ have the same moments up to
order $n$:
\begin{equation}
  \E\lrpar{\EUrho+\rho}^m = \E S_{n+1}^m,
\qquad 1\le m\le n.
\end{equation}
\end{lemma}

\begin{proof}
  Since $\EUrho+\rho\eqd\znrho+\rho=(S_n)_{1-\rho}$ by 
\eqref{znrho175}, this follows by \cite[Theorem 2.3]{SJ175}, noting that
\begin{equation}
  \tilde\gf(t)\=\frac{\eit-1}{\ii t}\gf_{S_n}(t)=\parfrac{\eit-1}{\ii t}^{n+1}
\end{equation}
has its $n$ first derivatives $0$ at $t=2\pi n$, $n\in\bbZ\setminus\set0$.
(Alternatively and equivalently, this follows from \eqref{efmgf}, noting
that the terms with $k\neq0$ give no contribution to the $m$:th derivative
at $t=0$ for $m\le n$, since the factor $(e^t-1)^{n+1}$ vanishes to order
$n+1$ there.)
\end{proof}

This leads to the following results, shown by
\citet[Lemmas 4.1 and 4.2]{GawronskiN} by related but more complicated
calculations. 
(In the classical case $\rho=1$, the cumulants were given already by 
\citet[p.~153]{DavidB}.) 

\begin{theorem}\label{Tmom}
  For any $\rho\in\bbR$,
  \begin{equation}\label{tmom1}
\E\EUrho = \frac{n+1}2-\rho, \qquad n\ge1,
  \end{equation}
and
  \begin{equation}\label{tmom2}
\Var\EUrho = \frac{n+1}{12}, \qquad n\ge2.
  \end{equation}
More generally, for $2\le m\le n$, the $m$:th cumulant $\kk_m(\EUrho)$
is independent of $\rho$, and is given by
  \begin{equation}\label{tmomkk}
\kk_m(\EUrho) = 
\kk_m(S_{n+1})
=(n+1)\kk_m(U)=(n+1)\frac{B_m}m,
\qquad 2\le m\le n,
  \end{equation}
where $B_m$ is the $m$:th Bernoulli number.
In particular, $\kk_m(\EUrho)=0$ if $m$ is odd with $3\le m\le n$.
\end{theorem}
For Bernoulli numbers, see \eg{} \cite[Section 6.5]{CM} and
\cite[\S24.2(i)]{NIST}. 

\begin{proof}
  For the mean we have by \refL{Lmom}, for $n\ge1$,
  \begin{equation}
\E\EUrho+\rho= \E S_{n+1} = (n+1)\E U = \frac{n+1}2,
  \end{equation}
which gives \eqref{tmom1}.

For higher moments, we note that the $m$:th cumulant $\kk_m$ can be
expressed as a polynomial in moments of order at most $m$; hence \refL{Lmom}
implies that $\kk_m(\EUrho+\rho)=\kk_m(S_{n+1})$ for $m\le n$.
Moreover, for any random variable (with $\E|X|^m<\infty$) and any real
number $a$, $\kk_m(X+a)=\kk_m(X)$, since $\kk_m(X+a)$ is the $m$:th
derivative at 0 of 
$\log \E e^{t(X+a)} = at + \log \E e^{tX}$. Hence,
\begin{equation}
  \kk_m(\EUrho)
=   \kk_m(\EUrho+\rho)
=\kk_m(S_{n+1}),
\qquad 2\le m\le n.
\end{equation}
Furthermore, since $S_{n+1}$ is the sum of the $n+1$ \iid{}
random variables $U_i$, $i\le n+1$, $\kk_m(S_{n+1})=(n+1)\kk_m(U)$.
Finally, 
the cumulants of the uniform distribution are
given by $\kk_m(U)=B_m/m$ for $m\ge2$; this, as is well-known, 
is shown by a simple calculation: (for $|t|<2\pi$;
see \cite[\S24.2.1]{NIST} for the last step)
\begin{equation*}
  \begin{split}
\summ m{\kk_m(U)}\frac{t^m}{m!}
&= 
t\ddq{t}\summ {\kk_m(U)}\frac{t^m}{m!}= 
t\ddq t \log \E e^{tU} =
t\ddq t \log \frac{e^t-1}{t}
\\&
=t\frac{e^t}{e^t-1}-1
=\frac{t}{e^t-1}+t-1
=\summ B_m\frac{t^m}{m!} +t-1.
  \end{split}
\end{equation*}
Combining these facts gives \eqref{tmomkk}. The special case $m=2$ yields
\eqref{tmom2}, since $\kk_2$ is the variance and $B_2=1/6$.
\end{proof}

We note also that \refT{T3} and
Fourier inversion for the distribution of $S_{n+1}$
yield the following integral formula.
(This is well-known in the settings of $S_n$,  and also for splines, see
\eg{} \cite[Theorem 3]{Schoenberg46}.) 
The case $\rho=0$ is given by \cite{Nicolas},
see also \cite{WangXuXu}.
(This integral formula is used in \cite{Nicolas} 
to define an extension $A(n,x)$ of Eulerian numbers to real $x$,
which by \eqref{tinta} equals the \EFx number $A_{n,\floor
  x,\frax x}$.
The formula is  further extended to real $n$ in \cite{LesieurN}; see also
\refR{Rfun}.)  
The special cases \eqref{tintb}--\eqref{tintc} are given by
\cite{GawronskiN}. 

\begin{theorem}
  \label{Tint}
If $n\ge1$, $k\in\bbZ$ and $\rho\in\oi$, then
\begin{equation}\label{tinta}
  \begin{split}
A_{n,k,\rho} 
&= \frac{n!}{\pi}\intoooo e^{\ii t(2k+2\rho-n-1)}
 \Bigparfrac{\sin t}{t}^{n+1}\,\dd t  
\\&
= n!\,\frac{2}{\pi}\intoo \cos\bigpar{t(2k+2\rho-n-1)}
 \Bigparfrac{\sin t}{t}^{n+1}\,\dd t  .	
  \end{split}
\end{equation}
In particular, for $k\ge1$, 
\begin{align}
  A_{2k-1,k,0} 
&= (2k-1)!\,\frac{2}{\pi}\intoo \Bigparfrac{\sin t}{t}^{2k}\,\dd t  ,
\label{tintb}
\\
  A_{2k,k,1/2} 
&= (2k)!\,\frac{2}{\pi}\intoo \Bigparfrac{\sin t}{t}^{2k+1}\,\dd t  .
\label{tintc}
\end{align}
\end{theorem}

\begin{proof}
  By \refT{T3}, Fourier inversion using \eqref{chfsn}, and finally
replacing $t$ by $2t$,
  \begin{equation*}
	\begin{split}
\frac{A_{n,k,\rho}}{n!}
&=f_{n+1}(k+\rho)	  
=\frac1{2\pi}\intoooo e^{-\ii t(k+\rho)} \gf_{S_{n+1}}(t)\,\dd t
\\&
=\frac1{2\pi}\intoooo e^{-\ii t(k+\rho)} \Bigparfrac{e^{\ii t}-1}{\ii t}^{n+1}
\,\dd t
\\&
=\frac1{\pi}\intoooo e^{\ii t(n+1-2k-2\rho)} \Bigparfrac{\sin t}{t}^{n+1}
\,\dd t
	\end{split}
  \end{equation*}
and the result follows by replacing $t$ by $-t$,  noting that $\sin t/t$
is an even function.
\end{proof}

\section{Asymptotic normality and large deviations}\label{SCLT}

It is well-known that the Eulerian distribution $\EUo$ or $\EUi$ is
asymptotically normal, and that furthermore a local limit theorem holds, \ie,
the Eulerian numbers $A_{n,k,1}=A_{n,k+1,0}$ 
can be approximated by a Gaussian function
for large $n$.
This has been  proved by various authors using several different metods, see
below; most of the methods generalize to $\EUrho$ and $A_{n,k,\rho}$ for
arbitrary $\rho\in\oi$.
The basic central limit theorem for $\EUrho$ can be stated as follows.
(Recall that the mean and variance are given by \refT{Tmom}.)
\begin{theorem}
  \label{TCLT}
$\EUrho$ is asymptotically normal as \ntoo, for any real $\rho$, \ie,
  \begin{equation}
	\frac{\EUrho-\E\EUrho}{(\Var\EUrho)\qq}\dto N(0,1)
  \end{equation}
or, more explicitly and simplified,
  \begin{equation}
	\frac{\EUrho-n/2}{n\qq}\dto N\Bigpar{0,\frac1{12}}.
  \end{equation}
\end{theorem}

\begin{proof}
  Immediate by \eqref{znrho} and the central limit theorem for $S_n=\sumiin
  U_i$.
Alternatively, the theorem follows by the method of moments from the formula
\eqref{tmomkk}
for the cumulants in 
\refT{Tmom},
which implies that the normalized cumulants
$\kk_m(\EUrho/(\Var\EUrho)\qq)$ converge to 0 for any $m\ge3$.
Several other proofs are described below.
\end{proof}
 
One refined verison is the
following local limit theorem with an asymptotic expansion
proved by \citet{GawronskiN};
we let as in \eqref{tmomkk} $B_m$ denote the Bernoulli numbers and let
$H_m(x)$
denote the Hermite polynomials, defined \eg{} by
\begin{equation}
  H_m(x)\=(-1)^me^{x^2/2}{\frac{\ddd^m}{\ddd x^m}}e^{-x^2/2},
\end{equation}
see \cite[p.~137]{Petrov}.
(These are the orthogonal polynomials for the standard normal distribution,
see \eg{} \cite{SJIII}.)

\begin{theorem}[\citet{GawronskiN}] 
\label{TLLT}
There exist polynomials $q_\nu$, $\nu\ge1$, such that, for any $\ell\ge0$,
as \ntoo, uniformly for all $k\in\bbZ$ and $\rho\in\oi$,
  \begin{equation}\label{tllt}
\frac{A_{n,k,\rho}}{n!}
=\sqrt{\frac{6}{\pi(n+1)}}\,e^{-x^2/2}
\lrpar{1+\sum_{\nu=1}^{\ell}\frac{q_\nu(x)}{(n+1)^\nu}} 
+O\lrpar{n^{-\ell-3/2}},	
  \end{equation}
where
\begin{equation}
  x=\Bigpar{k+\rho-\frac{n+1}2}\sqrt{\frac{12}{n+1}}.
\end{equation}
Explicitly, 
\begin{equation}
  q_\nu(x)=12^\nu\sum H_{2\nu+2s}(x) 6^{s}
\prod_{m=1}^\nu\frac1{k_m!}\parfrac{B_{2m+2}}{(m+1)(2m+2)!}^{k_m},
\end{equation}
summing over all non-negative integers $(k_1,\dots,k_{\nu})$ with
$k_1+2k_2+\dots+\nu k_\nu=\nu$ and letting $s=k_1+\dots+k_\nu$.
\end{theorem}

The polynomial $q_\nu(x)$ has degree $4\nu$.
The first two are:
\begin{align}
  q_1&=-\frac1{20}H_4(x) = - \frac{x^4-6x^2+3}{20},
\\
q_2&=\frac1{800}H_8(x)+\frac{1}{105}H_6(x)
=\frac{21x^8-428x^6+2010x^4-1620x^2-195}{16800}.
\end{align}

\begin{remark}
  We state \refT{TLLT} using an expansion in negative powers of $n+1$. Of
course, it is possible to use powers of $n$ instead, but then the polynomials
$q_i(x)$ will be modified.
\end{remark}

Before giving a proof of \refT{TLLT}, we give some history and discuss
various methods.
The perhaps first proof of asymptotic normality 
for Eulerian numbers
(\ie, \refT{TCLT} in the classical case $\rho=1$)
was given by
\citet{DavidB}, using the generating function \eqref{gf1} below to calculate
cumulants. 
\citet[Example 3.5]{Bender} used 
(for the equivalent case  $\rho=0$) 
instead a singularity analysis of the generating
function \eqref{gf1} to obtain this and further results;
see also
\citet[Example IX.12, p.~658]{FlajoletS}.

The representation \eqref{pb} of $\EUrho$ as a sum of independent 
(but not identically distributed)
Bernoulli
variables was  used by \citet{Carlitz-asN} to show asymptotic normality 
(global and local central limit theorems and a Berry--Esseen estimate) 
for Eulerian numbers
(\ie{} for $\rho=0$).
\citet{Sirazdinov} gave (also for $\rho=0$) 
a local limit theorem including the second order term ($\nu=1$)
in \refT{TLLT}.
(We have not been able to obtain the original reference, but we believe he
used this representation.)
More recently,  
\citet{GawronskiN} have used this method for a general $\rho\in\oi$
to show a global central limit theorem 
and
 the refined local limit theorem \refT{TLLT} above.

Furthermore, \eqref{pb} was also used (for $\rho=0$)
by \citet{Bender} to obtain a local limit theorem from the global central
limit theorem (proved by him by other methods, as said above).

Tanny \cite{Tanny} 
showed  global and local central limit theorems (for $\rho=0$) using the
representation $\EUo\eqd\floor{S_n}+1$, see our Theorems \ref{T1} and \ref{T4},
together with the standard central limit theorem for $S_n$;
see also \citet[Section 1.3.2]{Sachkov}.
This too extends to arbitrary $\rho$.

Esseen \cite{Esseen} used instead (still for $\rho=0$) the relation given
here as \eqref{t3} with the density function of $S_{n+1}$, together with the
standard local limit theorem for $S_n$.
This too extends to arbitrary $\rho$, and is perhaps the simplest method to
obtain local limit theorems. 
Moreover, as noted by \cite{Esseen}, it easily yields
an asymptotic expansion with
arbitrary many terms as in \refT{TLLT}. 
(For $\rho=0$, \cite{Esseen} gave the second term explicitly;
as mentioned above, this term was also given by \citet{Sirazdinov}.)
Furthermore, the first three terms 
($\nu\le2$)
in \refT{TLLT}
were given (for arbitrary $\rho$) by \citet{Nicolas},
using essentially the same method, but stated in analytic formulation
rather than probabilistic. (See also, for the case $\rho=0$, \cite{XuWang}.)

Esseen \cite{Esseen} also pointed out that 
$\EUi=\EUo-1$, regarded as the number of descents in a random permutation,
can be represented as
\begin{equation}
\EUi\eqd \sum_{i=1}^{n-1}\ett{U_i>U_{i+1}},  
\end{equation}
where $U_1,U_2,\dots$ are \iid{} with
the distribution $U(0,1)$; the global central limit theorem thus follows
immediately from the standard central limit theorem for $m$-dependent
sequences.
(However, we do not know any similar representation for the case
$\rho\in(0,1)$.) 

\begin{proof}[Proof of \refT{TLLT}]
We use the method of  \cite{Esseen}, and note that
the theorem follows immediately from \eqref{t3} and the local limit theorem for
$f_n(x)$, see \eg{} \cite[Theorem VII.15 and (VI.1.14)]{Petrov},
using $\kk_m(U_i)=B_m/m$ for $m\ge2$ and noting that $B_m=0$ for odd
$m\ge3$.
(As said above, the proof in \cite{GawronskiN} is somewhat different,
although it also uses \cite{Petrov}.)
\end{proof}

\begin{remark}\label{Rtail}
By using \cite[Theorem VII.17]{Petrov} in the proof of \refT{TLLT}
(and taking as many terms as needed),
the error term in \eqref{tllt} is improved to 
$O\lrpar{(1+|x|^K)\qw n^{-\ell-3/2}}$, for any fixed $K>0$.	
This is, however, a superficial improvement, since this is trivial for
$k+\rho\notin[0,n+1]$ (when $A_{n,k,\rho}=0$), and otherwise easily follows 
from \eqref{tllt} by increasing $\ell$.
\end{remark}

Although \refT{TLLT} holds uniformly for all $k$, it is of interest mainly
when $k=n/2+O(\sqrt{ n\log n})$, when $x=O(\sqrt{\log n})$, since 
for  $|k-n/2|$  larger, the main term in \eqref{tllt} is  smaller
than the error term. (This holds also for the improved version in \refR{Rtail}.)
However, 
we can also easily obtain asymptotic estimates for other $k$ by the same method,
now using \eqref{t3} and
large deviation estimates for the density $f_{n+1}(x)$, 
which are obtained by  standard arguments.
(The saddle point method, which in this context is essentially the same as 
using Cram\'er's \cite{Cramer} method of conjugate distributions.
See \eg{} \cite[Chapter 2]{DemboZeitouni} or \cite[Chapter 27]{Kallenberg}
for more general large deviation theory, and \cite{FlajoletS} for the saddle
point method.)
In the classical case $\rho=0$,
this was done 
by \citet{Esseen} 
(for $1\le k\le n$),
improving an earlier result
by \citet{Bender} ($\eps n \le k \le(1-\eps)n$ for any $\eps>0$)
who used a related argument using the generating function \eqref{gf1}.
The result extends immediately to any $\rho$
as follows. 

Let $\psi(t)$ be the moment generating function of $U\sim U(0,1)$, \ie,
\begin{equation}
  \psi(t)=\frac{e^t-1}{t},
\end{equation}
and let for $a\in(0,1)$ 
\begin{equation}\label{ma}
  m(a)\=\min_{-\infty<t<\infty} e^{-at}\psi(t).
\end{equation}
Since $\log\psi(t)$ is convex (a general property of \mgf{s}, and easily
verified directly in this case), and the derivative $(\log\psi)'$
increases from 0 to 1, the minimum is attained at a unique
$t(a)\in(-\infty,\infty)$ for each $a\in(0,1)$, which is the solution to the
equation 
\begin{equation}\label{at}
  a=(\log\psi)'(t) = \frac{e^t}{e^t-1}-\frac{1}t
= \frac{1}{1-e^{-t}}-\frac{1}t.
\end{equation}
Set
\begin{equation}
  \gss(a)\=(\log\psi)''(t(a)) = \frac{1}{t(a)^2}-\frac{e^{t(a)}}{(e^{t(a)}-1)^2}
=\frac{1}{t(a)^2}-\frac{1}{\sinh^2 t(a)},
\end{equation}
interpreted (by continuity) as $(\log\psi)''(0)=1/12$ when $a=1/2$ and thus
$t(a)=0$. 

\begin{theorem}\label{Tldev}
Assume  $\rho\in\oi$ and $0< k+\rho< n+1$.
Then
  \begin{equation}\label{tldev}
	\frac{A_{n,k,\rho}}{n!}
=\frac{(m(a))^{n+1}}{\sqrt{2\pi(n+1)\gss(a)}}\bigpar{1+O(n\qw)},
  \end{equation}
where   $a\=(k+\rho)/(n+1)$, uniformly in all $k$ and $\rho$ such that
$0<  k+\rho< n+1$.
\end{theorem}

\begin{proof}
We follow \citet{Esseen} with some additions.
  By \eqref{t3} (and replacing $n$ by $n-1$), 
the statement is equivalent to the estimate
  \begin{equation}\label{maj}
f_n(x)
=\frac{(m(a))^{n}}{\sqrt{2\pi n\gss(a)}}\bigpar{1+O(n\qw)},
  \end{equation}
where $a=x/n$, uniformly for $x\in(0,n)$.
To see this, we use Fourier inversion (assuming $n\ge2$), 
noting that the \chf{} of $S_n$ is
$\psi(\ii t)^n$,  and shift the line of integration to the saddle point $t(a)$:
\begin{equation}
  \begin{split}
  f_n(x)
&=
\frac{1}{2\pi}\int_{-\infty}^{\infty} e^{-\ii xt}\psi(\ii t)^n\dd t
=\frac{1}{2\pi\ii}\int_{-\infty\ii}^{\infty\ii} \bigpar{e^{-az}\psi(z)}^n\dd z
\\&
=\frac{1}{2\pi\ii}\int_{t(a)-\infty\ii}^{t(a)+\infty\ii} 
\bigpar{e^{-az}\psi(z)}^n\dd z
=\frac1{2\pi}\intoooo g\bigpar{t(a)+iu}^n\dd u,
  \end{split}
\end{equation}
where $g(z)\=e^{-az}\psi(z)$. (Thus $g(t(a))=m(a)$.)

If we assume $x\in[1,n-1]$ and thus $a\in[1/n,1-1/n]$, then
$t(a)=O(n)$ by \eqref{at}.
Routine estimates (and the change of variable
$u=s/\sqrt n$) show that then the integral over $|u|\le n^{-.49}$ 
yields the \rhs{} of \eqref{maj}, uniformly in $x$,
while the remaining integral is smaller by a factor $O(n^{-100})$ 
(for example). 
We omit the details.

If $0<x<1$, so $0<a<1/n$,
then $f_n(x)=x^{n-1}/(n-1)!$, while \eqref{at} and \eqref{ma} yield
$t(a)=-a\qw+O(a^{-2}e^{-1/a})=-n/x+O(e^{-n/2x})$ 
and 
\begin{equation}
  m(a)=e^{-at(a)}\psi(t(a)) = ae^{-1}\bigpar{1+O(e^{-n/2})};
\end{equation}
then \eqref{tldev} is easily verified directly,
using Stirling's formula. The case $n-1<x<n$ is symmetric.
We again omit the details.
\end{proof}

As remarked by \citet{Esseen}, it is possible to obtain an expansion with
further terms in \eqref{tldev} by the same method.

The saddle point method is standard in problems of this type. However, it is
perhaps surprising that it can be used with a uniform relative error bound
for all $t(a)\in(-\infty,\infty)$.

\section{More on rounding}
\label{Sval}

Let $p_1,\dots,p_n$ be given probabilities (or proportions) with 
$\sumiin p_i=1$ and let $N$ be a (large) integer.
Suppose that we want to round $Np_i$ to integers, 
by standard rounding or, more generally, by \rr{} for some given $\rho\in\bbR$.
It is often desirable that the sum of the results is exactly $N$, but this
is, of course,  not always the case. We thus consider the discrepancy
\begin{equation}
\gdr\=\sumiin\rrr{Np_i}-N.  
\end{equation}
More generally, we may round $(N+\gam)p_i$ for some fixed $\gam\in\bbR$
and we define the discrepancy
\begin{equation}\label{gdrgg}
\gdrgg\=\sumiin\rrr{(N+\gam)p_i}-N.  
\end{equation}

\begin{example}
If a statistical table is presented as percentages,
with all percentages rounded to integers, we have this situation with $N=100$;
the percentages do not necessarily add up to 100, and the error is given by
$\gdx{1/2}$. Rounding to other accuracies correspond to other values of $N$.
This has been studied by several authors, see
\citet{MostellerYZ} and
\citet{DiaconisF}.
\end{example}

\begin{example}\label{Eval7}
The general idea of a proportional election method is that a given number
$N$ of seats are to be distributed among $n$ parties which have obtained
$v_1,\dots,v_n$ votes. With $p_i\=v_i/\sum_{j=1}^n v_j$, the proportion of
votes for party $i$, the party should ideally get $Np_i$ seats, but the
number of seats has to be an integer so some kind of rounding procedure is
needed. (See \refE{Eval} and the reference given there
for some important methods used in practice.)

A simple attempt would be to round $Np_i$ to the nearest integer and give
$\rrx{1/2}{Np_i}$ seats to party $i$.
More generally, we might fix some $\rho\in\oi$ and
give the party  $\rrr{Np_i}$ seats.
Of course, this is not a workable election
method since the
sum in general is not exactly equal to $N$, and the error is given by
$\gdr$. 
(In principle the method could be used for elections if one accepts a
varying size of the elected house, but I don't know any examples of it being
used.) 
Nevertheless, this can be seen as the first step in an algorithm
implementing divisor methods, see \citet{HappacherP96} and \citet{Pukelsheim}.
In this context it is also useful to consider the more general
$\rrr{(N+\gam)p_i}$ for some given $\gam\in\bbR$, 
see again \cite{HappacherP96} and \cite{Pukelsheim}; then the error is given
by \eqref{gdrgg}.
\end{example}

\begin{example}
Roundings of $(N+\gam)p_i$ and thus \eqref{gdrgg} occur also in the study of
quota methods of elections, for example  
\emph{Droop's method} where we take $\gam=1$ and adjust $\rho$ to obtain the
sum $N$, see \cite[Appendix B]{SJ262}.
\end{example}

If we assume that the proportions $p_i$ are random, it is thus of interest
to find the distribution of the discrepancy $\gdrgg$, and in particular
of $\gdx{\rho,0}=\gdr$. 
We consider the asymptotic distribution
of $\gdrgg$ as $N\to\infty$. 

The simplest assumption is that $(p_1,\dots,p_n)$ is uniformly distributed
over the $n-1$-dimensional unit simplex
$\set{(p_i)_1^n\in\bbR_+^n:\sum_ip_i=1}$, but it 
turns out (by Weyl's lemma, see \eg{} \cite[Lemma 4.1 and Lemma C.1]{SJ262})
that the asymptotic distribution is the same for any absolutely continuous
distribution of $(p_i)_1^{n-1}$, and we state our results for this setting. 

\begin{remark}
Alternatively, it is also possible to consider a
fixed $(p_i)_1^n$ (for almost all choices) and let $N$ be random as in
\cite[Section 1]{SJ262} (with $N\pto\infty$). The same asymptotic
results are obtained in this case too, but
we leave the details to the reader.
\end{remark}

In the standard case $\rho=1/2$ and $\gam=0$, 
the asymptotic distribution of the discrepancy $\gdx{1/2}$ 
was found by \citet{DiaconisF}, assuming (as we do here) that $(p_i)_1^n$
have an absolutely continuous distribution on the unit simplex.
(The cases $n=3,4$,
with uniformly distributed probabilities $(p_i)_1^n$,
were earlier treated by
\citet{MostellerYZ}.)
This was extended to $\gdr$ with arbitrary $\rho$ (still with $\gam=0$)
by \citet{BalinskiR}. 
\citet{HappacherP96,HappacherP00} considered general $\rho$ and $\gam$
(assuming  uniformly distributed $(p_i)_1^n$)
and found 
asymptotics of the mean and variance of $\gdrgg$ 
\cite{HappacherP96} and 
(at least in the case $\gam=n(\rho-\frac12)$)
the asymptotic distribution
\cite{HappacherP00}. 
The exact distribution of $\gdrgg$ for a finite $N$
(assuming uniform  $(p_i)_1^n$)
was given by
\citet[Sections 2 and 3]{Happacher}. 
Furthermore, asymptotic results for the
probability $\P(\gdrgg=0)$ of no discrepancy
have also been given 
(assuming uniform  $(p_i)_1^n$)
by
\citet[p.~185]{Kopfermann} ($\rho=1/2$, $\gam=0$), 
and
\citet{GawronskiN} ($\rho=1/2$, $\gam=0$); the latter paper moreover gives the
connection to \EFx numbers. (The other papers referred to here
state the results using $S_{n-1}$ or the density function $f_n$ of $S_n$, or
the explicit sum in \eqref{fn} for the latter.)

We state and extend these asymptotic results for $\gdrgg$ as follows.
Equivalent reformulations of \eqref{troundb}--\eqref{trounda} 
(including the versions in the references above)
can be given
using
\eqref{t4b}, see also \eqref{qa} below.
\begin{theorem}\label{Tround}
  Suppose that $(p_1,\dots,p_n)$ are random with an absolutely continuous
  distribution on the $n-1$-dimensional unit simplex. Then, as $N\to\infty$,
for any fixed real $\rho$ and $\gam$, $n\ge2$ and $k\in\bbZ$,
\begin{equation}\label{troundb}
  \gdrgg\dto
\EU_{n-1,n\rho-\gam}.
\end{equation}
In other words,
\begin{equation}\label{trounda}
  \P(\gdrgg=k)
\to
\frac{A_{n-1,k+\floor{n\rho-\gam},\frax{n\rho+\gam}}}{(n-1)!}.
\end{equation}

Furthermore, 
\begin{align}
\E  \gdrgg&\to 
\E\EU_{n-1,n\rho-\gam}=
n\Bigpar{\frac12-\rho}+\gam \label{troundc}
\intertext{and if $n\ge3$,}
\Var  \gdrgg&\to 
\Var\EU_{n-1,n\rho-\gam}=
\frac{n}{12}. \label{troundd}
\end{align}
\end{theorem}

\begin{proof}
  Let $X_i\=(N+\gam)p_i+1-\rho$ and $Y_i\=\frax{X_i}$.
Thus, by the definition \eqref{rrr}, $\rrr{(N+\gam)p_i}=\floor{X_i}$ and
hence by \eqref{gdrgg},
  \begin{equation}\label{tr0}
	\begin{split}
\gdrgg&
=\sum_{i=1}^n\floor{X_i}-N	  
=\sum_{i=1}^n(X_i-Y_i)-N
\\&
=N+\gam+n(1-\rho)-\sumiin Y_i - N
\\&
=\gam+n(1-\rho)-\sumiin Y_i.	  
	\end{split}
  \end{equation}
Since this is an integer, and $Y_{n}\in[0,1)$, it follows that
  \begin{equation}\label{tr1}
\gdrgg
=\lrfloor{\gam+n(1-\rho)-\sumiini Y_i}.	  
  \end{equation}
As $N\to\infty$, the fractional parts 
$(\frax{Np_i})_1^{n-1}$ converge jointly
in distribution to the
independent uniform random variables $(U_i)_1^{n-1}$, 
see \cite[Lemma C.1]{SJ262},
and 
since $Y_i=\frax{\frax{Np_i}+\gam p_i+1-\rho}$, 
the same holds for $(Y_i)_{1}^{n-1}$.
Thus \eqref{tr1} implies, using $1-U_i\eqd U_i$,
\begin{equation}\label{qa}
  \begin{split}
\gdrgg&\dto
\lrfloor{\gam+n(1-\rho)-\sumiini U_i}
\eqd\lrfloor{\gam+1-n\rho+\sumiini U_i}
\\&
=\lrfloor{\gam+1-n\rho+S_{n-1}}
=\rrx{n\rho-\gam}{S_{n-1}}
.
  \end{split}
\end{equation}
Since 
$\rrx{n\rho-\gam}{S_{n-1}}
\eqd \EU_{n-1,n\rho-\gam}$
by \refT{Trr}, this proves \eqref{troundb}; furthermore, \refT{Trr} yields also 
\eqref{trounda}.
Since $\gdrgg$ is uniformly bounded by
\eqref{tr0}, \eqref{troundb} implies moment convergence  and thus
\eqref{troundc}--\eqref{troundd} follow by \refT{Tmom}.
\end{proof}

\begin{remark}
In the case of uniformly distributed probabilities $(p_i)_1^n$, 
\citet{HappacherP96} 
have given a more precise form of
the moment asymptotics \eqref{troundc}--\eqref{troundd} 
with explicit higher order terms.
\end{remark}

\begin{example}
We see from \eqref{troundc} that the choice
$\gam=n(\rho-\frac12)$ yields $\E\gdrgg\to0$, so the discrepancy is
asymptotically unbiased. This was shown by \citet{HappacherP96}, who therefore
recommend this choice of $\gam$ when the aim is to try to avoid a discrepancy;
see also \cite{HappacherP00}.  
\end{example}

For standard rounding, $\rho=1/2$ and $\gam=0$, 
\refT{Tround} yields the result of \citet{DiaconisF}, which we state as a
corollary.

\begin{corollary}\label{Cr}
With assumptions as in \refT{Tround}, 
\begin{equation}\label{cr1}
  \gdqq\dto
\EU_{n-1,n/2}.
\end{equation}
In other words,
\begin{equation}\label{cr2}
  \P(\gdqq=k)
\to
\frac{A_{n-1,k+\floor{n/2},\frax{n/2}}}{(n-1)!}.
\end{equation}
Using the notation of \refC{C1}, the result can also be written
\begin{equation}\label{cr3}
  \gdqq\dto
\floor{\tS_{n-1}}_{1/2}.
\end{equation}
\end{corollary}

\begin{proof}
  \refT{Tround} yields \eqref{cr1}--\eqref{cr2}, and \eqref{cr3} then
  follows by \refC{C1}.
\end{proof}

The asymptotic distribution in \refC{Cr} can also be described as
$\EU_{n-1,0}$ when $n$ is even and $\EU_{n-1,1/2}$ when $n$ is
odd, in both cases centred by subtracting the mean $\floor{n/2}$.
Note that in this case the asymptotic distribution is symmetric, \eg{} by
\eqref{EUsymm}. 

\begin{example}
Taking  $k=0$ in \refC{Cr} we obtain the asymptotic
probability that standard rounding yields the correct sum:
\begin{equation}\label{e7}
  \P(\gdx{\xfrac12}=0)
\to
A_{n-1,\floor{n/2},\frax{n/2}}/(n-1)!.
\end{equation}
The limit is precisely the value for
$\P\bigpar{\floor{\tS_{n-1}}_{1/2}=0}$ given by \refC{C1}.
\end{example}

\begin{remark}
  Even if $(p_1,\dots,p_n)$ is uniformly distributed, the result in 
\refT{Tround}
is in general only asymptotic and not exact for finite $N$ because of edge
effects. For a simple example, if $\rho=1/2$ and $N=1$, then 
$\P(\gdx{\xfrac12}=0)=\P\bigpar{\sum_1^n\rrx{\xfrac12}{p_i}=1}
=n \P\bigpar{p_1>1/2} = n2^{1-n}$, which differs from the asymptotical value
in \eqref{e7} for $n\ge4$. 
(It is much smaller for large $n$, since it
decreases exponentially in $n$.)
See \cite{Happacher} for an exact formula for finite $N$.
\end{remark}
 
\section{Example: The method of greatest remainder}\label{Share}

By \eqref{t4c}, \eqref{efreal} and \eqref{efdist}, the distribution function
$F_n$ of $S_n$ can be expressed using the \EFx distribution
or using \EFx numbers.
As another example involving election methods, 
consider again an election as in \refE{Eval7}, with $N$ seats distributed among
$n\ge3$ parties having proportions $p_1,\dots,p_n$ of the votes. 
Let 
$s_1,\dots,s_n$ be the number of seats assigned to
the parties by the method of greatest remainder (Hare's method; Hamilton's
method), see \refE{Eval}, and  consider the bias $\xgD_1\=s_1-Np_1$ for party
1. 
Let again $(p_1,\dots,p_n)$ be random as in \refT{Tround}, and let
$N\to\infty$;
or let $(p_1,\dots,p_n)$ be fixed and $N$ random, with conditions
as in \cite{SJ262}.
It is shown in \cite[Theorems 3.13 and 6.1]{SJ262} that then 
\begin{equation}\label{job}
 \xgD_1\dto \hgd\=\tU_0+\frac1n\sum_{i=1}^{n-2}\tU_i = \tU_0+\frac1n\tS_{n-2},
\end{equation}
where $\tU_i\=U_i-\frac12$ are independent and uniform on $[-\frac12,\frac12]$.

The limit distribution in \eqref{job} has density function, 
using 
\eqref{t4c} and \eqref{efreal},
\begin{equation}\label{qme}
  \begin{split}
\ftd(x)&=\P\Bigpar{\frac1n\tS_{n-2}\in\bigpar{x-\tfrac12,x+\tfrac12}}	
\\&
=\P\bigpar{\tS_{n-2}\in\xpar{nx-\xfrac n2,nx+\xfrac n2}}	
\\&
=\P\bigpar{S_{n-2}\in\xpar{nx-1,nx+n-1}}	
\\& 
=F_{n-2}(nx+n-1) - F_{n-2}(nx-1)
\\&
=\P\bigpar{0\le\EU_{n-2,nx}\le n-1}
\\&
=\P\bigpar{\floor{nx}\le  \EU_{n-2,\frax{nx}} < \floor{nx}+n}.
  \end{split}
\end{equation}
This can also by \eqref{efdist} be expressed in the \EFx numbers: 
\begin{equation}\label{qma}
\ftd(x)
=\frac1{(n-2)!}\sum_{j=0}^{n-1}A_{n-2,\floor{nx}+j,\frax{nx}}.
\end{equation}

By \eqref{job}, $|\hgd|\le 1-1/n<1$; in particular, $\ftd(x)=0$ for
$x\notin(-1,1)$, as also can be seen from \eqref{qme} or \eqref{qma}.
Note also that, for every $x$,
\begin{equation}
  \sum_{i=-\infty}^\infty \ftd(x+i) 
=  \sum_{i=-\infty}^\infty 
\P\bigpar{\floor{nx}+ni\le  \EU_{n-2,\frax{nx}} < \floor{nx}+ni+n}
=1.
\end{equation}
(This means that $\frax{\hgd}$ is uniformly distributed on $\oi$, which
also easily is seen directly from \eqref{job}.)
It follows that, for all $x\in\oi$ and $k\in\bbZ$,
\begin{equation}\label{qj}
  \P\bigpar{\floor{\hgd}=k\mid\frax{\hgd}=x}
=\frac{\ftd(k+x)}{  \sum_{i=-\infty}^\infty \ftd(x+i) }=\ftd(k+x).
\end{equation}

Of course, the fractional part
\begin{equation}
  \frax{\xgD_1}=\frax{-Np_1}=1-\frax{Np_1}
\end{equation}
(unless $Np_1$ is an integer). Thus, for $x\in[0,1)$ and a small $dx>0$,
with $y\=1-x$,
\begin{equation*}
\frax{\xgD_1}\in(x,x+dx)\iff\frax{Np_1}\in(y-dx,y).   
\end{equation*}
Conditioned on this
event, for $k\in\bbZ$,
\begin{equation}\label{qn}
  \begin{split}
  \P\bigpar{\floor{\xgD_1}=k\mid \frax{\xgD_1}\in(x,x+dx)}
&= \frac{\P(\xgD_1\in(k+x,k+x+dx))}{\P(\frax{\xgD_1}\in(x,x+dx))}
\\&
\to \frac{\P(\hgd\in(k+x,k+x+dx))}{\P(\frax{\hgd}\in(x,x+dx))}
  \end{split}
\end{equation}
where the \rhs{} as $dx\to0$ converges to, see \eqref{qj},
\begin{equation}\label{qo}
  \P\bigpar{{\hgd}=k+x\mid\frax{\hgd}=x}
=\ftd(k+x).
\end{equation}
This leads to the following result.

\begin{theorem}
  Let $p_1\in(0,1)$ be fixed and suppose that $(p_2,\dots,p_n)$ have an
  absolutely continuous distribution in the simplex 
$\set{(p_i)_2^n\in\bbR_+^{n-1}:\sum_2^np_i=1-p_1}$, where $n\ge3$.
Then, for the method of greatest remainder, as $N\to\infty$, 
\begin{equation}\label{tglo}
  \begin{split}
  \P\bigpar{\xgD_1=k-\frax{Np_1}} 
&=\ftd\bigpar{k-\frax{Np_1}}+o(1)
  \end{split}
\end{equation}
with $\ftd$ given by \eqref{qme}--\eqref{qma},
for every $k\in\bbZ$.
Thus,
\begin{align*}
  \P\bigpar{\xgD_1=-\frax{Np_1}} 
&=\P\bigpar{ \EU_{n-2,\frax{-nNp_1}} \le n-\ceil{n\frax{Np_1}}}
+o(1),
\\
 \P\bigpar{\xgD_1=1-\frax{Np_1}} 
&=\P\bigpar{ \EU_{n-2,\frax{-nNp_1}} > n-\ceil{n\frax{Np_1}}}
+o(1).	
\end{align*}
\end{theorem}
The result \eqref{tglo} is trivial unless $k\in\setoi$, since the
probability 
is 0 otherwise.

\begin{proof}
  Let $X_i\=Np_i$. Now $X_1$ is deterministic, but the fractional
  parts $(\frax{X_i})_{i=2}^{n-1}$ converge to independent uniform random
  variables $(U_i)_{i=2}^{n-1}$ as in \refS{Sval}. 
The seat bias $\xgD_1$ depends only
  on the 
  fractional parts $\frax{X_i}$, and it is an \aex{} continuous function 
$h(\frax{X_1},\dots,\frax{X_n})$ of
  them, and it follows that for any subsequence of $N$
with $\frax{X_1}=\frax{Np_1}\to
  y\in[0,1]$, 
\begin{equation}
  \label{qi}
\xgD_1\dto
  Y(y)\=h(y,U_2,\dots,U_{n-1},\frax{-y-U_2-\dots-U_{n-1}}).
\end{equation} 
Moreover, the distribution of $Y(y)$ depends continuously on $y\in[0,1]$,
with $Y(0)=Y(1)$.

If we now would let $p_1$ be random and uniform in some small interval,
and scale $(p_2,\dots,p_n)$ correspondingly, and then condition on
$\frax{Np_1}\in(y-dx,y)$, for $y\in(0,1]$,
then \eqref{qn} would apply. The left hand side of
\eqref{qn} is asymptotically the average of $\P(\floor{Y(z)}=k)$ for
$z\in(x,x+dx)$ with $x=1-y$, and by the continuity of the distribution of
$Y(z)$, we can let $dx\to0$ and conclude from \eqref{qn} and \eqref{qo} that
$Y(y)$ has the distribution in \eqref{qo}, \ie,
\begin{equation}\label{qf}
  \P(Y(y)=k+1-y)=\ftd(k+1-y).
\end{equation}

If \eqref{tglo} would not hold, for some $k$, then there would be a
subsequence such that 
$\P\bigpar{\xgD_1=k-\frax{Np_1}}$ converges to a limit different from 
$\ftd\bigpar{k-\frax{Np_1}}$. We may moreover assume that $\frax{Np_1}$
converges to some $y\in\oi$, but then \eqref{qi} and \eqref{qf} would
yield a contradiction. This shows \eqref{tglo}.

The final formulas follow by taking $k=0,1$ in \eqref{tglo} and using
\eqref{qme}. 
\end{proof}

\begin{remark}
  For other quota methods,
\cite[Theorems 3.13 and 6.1]{SJ262} provide similar results. In particular,
for Droop's method, \eqref{job} is replaced by
\begin{equation}\label{job1}
 \xgD_1\dto p_1-\frac1n+\hgd,
\end{equation}
with $\hgd$ as in \eqref{job}, and thus the argument above shows that 
\eqref{tglo} is replaced by 
\begin{equation}\label{tglo1}
  \begin{split}
  \P\bigpar{\xgD_1=k-\frax{Np_1}} 
&=\ftd\bigpar{k-\frax{Np_1}-p_1+1/n}+o(1).
  \end{split}
\end{equation}
There are also similar results for divisor methods,
see \cite[Theorems 3.7 and 6.1]{SJ262}. In particular, for Sainte-Lagu\"e's
method, 
\begin{equation}
 \xgD_1\dto  \tU_0+p_1\tS_{n-2};
\end{equation}
this too can be expressed using \EFx distributions or
\EFx numbers as above, but
the result is more complicated and omitted.
\end{remark}

\appendix

\section{\EFx polynomials}\label{AEF}

We collect in this appendix some known facts for easy reference.
(Some are used above; others are included because we find them interesting
and perhaps illuminating.)
We note first that the sum in \eqref{ef} is absolutely convergent for
$|x|<1$, and that the second equality holds there (or as an equality of
formal power series).
The first equality in \eqref{ef} (which is valid for all complex $x\neq1$)
can be written
as a ``Rodrigues formula''
\begin{equation}
P_{n,\rho}(x)=
(1-x)^{n+1}\Bigpar{\rho+x\ddq{x}}^n\frac1{1-x}
,
\end{equation}
which yields the recursion
\begin{equation}
P_{n,\rho}(x)=
(1-x)^{n+1}\Bigpar{\rho+x\ddq{x}}\Bigpar{P_{n-1,\rho}(x)\xpar{1-x}^{-n}}
,
\qquad n\ge1;
\end{equation}
after expansion, this yields \eqref{efrec}.

We note that by the recursion \eqref{efrec} and induction
\begin{equation}
  A_{n,0,\rho}=\pnrho(0) =\rho^n,
\end{equation}
and
\begin{equation}\label{pnrho1}
  \sumkn A_{n,k,\rho}=\pnrho(1) =n!;
\end{equation}
moreover, by \eqref{efarec} and induction,
\begin{equation}
  A_{n,n,\rho} =(1-\rho)^n.
\end{equation}
In particular, if $\rho\neq1$, then $A_{n,n,\rho}\neq0$ and $\pnrho$ has
degree exactly $n$ for every $n$.

The case $\rho=1$ is special; in this case $A_{n,n,\rho}=0$ for $n\ge1$, so
$\pnrho$ has degree $n-1$ for $n\ge1$. In fact, if follows directly from
\eqref{ef} that
\begin{equation}\label{pn01}
  P_{n,0}(x)=xP_{n,1}(x),
\qquad n\ge1;
\end{equation}
hence, as said in the introduction,
the Eulerian numbers appear twice as
\begin{equation}\label{a01}
  A_{n,k+1,0}=A_{n,k,1}, \qquad n\ge1.
\end{equation}

A binomial expansion in \eqref{ef} shows that $\pnrho$ can be expressed
using the classical special case $\rho=0$ as
\begin{equation}\label{q21}
  \pnrho(x)=\sumin \binom ni \rho^{i}(1-x)^{i} P_{n-i,0}(x),
\end{equation}
which shows that $\pnrho(x)$, and thus also each $A_{n,k,\rho}$, is a
polynomial in $\rho$ of degree at most $n$. 
By \eqref{q21}, the $n$:th degree term is $(1-x)^n\rho^n$; hence the
degree in $\rho$ is exactly $n$ for $\pnrho(x)$ for any $x\neq1$ (recall that
$\pnrho(1)=n!$ does not depend on $\rho$) and for $A_{n,k,\rho}$ for any
$k=0,\dots,n$. (The leading term of $A_{n,k,\rho}$ is $(-1)^k\binom nk
\rho^n$.) 

Similarly, another binomial expansion in \eqref{ef} yields
\begin{equation}\label{q211}
  \begin{split}
\pnrho(x) & =\sumin \binom ni (\rho-1)^{i}(1-x)^{i} P_{n-i,1}(x)
\\&
 =\sumin \binom ni (1-\rho)^{i}(x-1)^{i} P_{n-i,1}(x).	
  \end{split}
\end{equation}
For $\rho=0$, this yields, using \eqref{pn01}, the recursion formula used by
\citet[(6.) p.~826]{Frobenius} to define the Eulerian polynomials.


The definition \eqref{ef} yields (and is equivalent to) the generating
function
\begin{equation}\label{gf}
  \sumn \frac{\pnrho(x)}{(1-x)^{n+1}}\frac{z^n}{n!}
=\sumn\sumj(j+\rho)^n x^j \frac{z^n}{n!}
=\sumj e^{jz+\rho z} x^j
=\frac{e^{\rho z}}{1-xe^z}
\end{equation}
or, equivalently,
\begin{equation}\label{gf1}
  \sumn{\pnrho(x)}\frac{z^n}{n!}
=\frac{(1-x)e^{\rho z(1-x)}}{1-xe^{z(1-x)}}.
\end{equation}
(See also \citet[pp.~150--152]{DavidB} and 
\citet[Example III.25,  p.~209]{FlajoletS} for the case $\rho=1$.) 
In particular, for the classical case $\rho=1$,
\eqref{gf} can be written
\begin{equation}\label{gf2}
  \frac{1-x}{e^z-x}=\frac{(1-x)e^{-z}}{1-xe^{-z}}
=  \sumn \frac{P_{n,1}(x)}{(1-x)^{n}}\frac{(-z)^n}{n!}
=  \sumn \frac{P_{n,1}(x)}{(x-1)^{n}}\frac{z^n}{n!},
\end{equation}
discovered by Euler \cite[\S 174, p.~391]{E212};
this is sometimes taken as a definition, see \eg{} 
\citet[p.~39]{Riordan} and
\citet{Carlitz}.
More generally, we similarly obtain, \cf{} \eqref{carl} and \cite{Carlitz},
\begin{equation}
\sumn \frac{P_{n,1-u}(x)}{(x-1)^n} \frac{z^n}{n!}
=e^{zu} \frac{1-x}{e^z-x}.
\end{equation}


From \eqref{gf} one easily obtains the symmetry relation
\begin{equation}
  \label{psymm}
x^n P_{n,\rho}(x\qw)=P_{n,1-\rho}(x),
\end{equation}
or equivalently, by \eqref{efa},
\begin{equation}
  \label{asymm}
A_{n,n-k,\rho}=A_{n,k,1-\rho},
\end{equation}
which also easily is proved by induction.
In terms of the random variables $\EUrho $ defined in \eqref{efdist}, this
can be written
\begin{equation}
  \label{EUsymm}
\EU_{n,1-\rho}\eqd n-\EU_{n,\rho}.
\end{equation}

\begin{remark}\label{RPsymm}
If we define the 
homogeneous  two-variable polynomials
  \begin{equation}\label{hP}
\hP_{n,\rho}(x,y)\=\sumkn \anrk  x^k y^{n-k}, 
\end{equation}
so that 
$\hP_{n,\rho}(x,y)=y^nP_{n,\rho}(x/y)$
and $P_{n,\rho}(x)=\hP_{n,\rho}(x,1)$,
the symmetry \eqref{psymm}--\eqref{asymm} takes the form
$\hP_{n,\rho}(x,y)=\hP_{n,1-\rho}(y,x)$. 
(In particular, in the classical case $\rho=1$, this together with
\eqref{pn01} shows that $y\qw\hP_{n,1}(x,y)$ is a symmetric homogeneous
polynomial of degree $n-1$ \cite{Frobenius}.)
The recursion \eqref{efrec} becomes
\begin{equation}
  \hP_{n,\rho}(x,y)
= \Bigpar{(1-\rho)x+\rho y
 +xy\frac{\partial}{\partial x}+xy\frac{\partial}{\partial y}}
\hP_{n-1,\rho}(x,y),
\qquad n\ge1.
\end{equation}
\end{remark}

For $x<1$, the sum in \eqref{ef} can be differentiated in $\rho$ termwise
(for all $\rho\in\bbC$), which yields, for $n\ge1$,
\begin{equation}
\frac{\partial}{\partial\rho}\frac{P_{n,\rho}(x)}{(1-x)^{n+1}}
=
\sumj n (j+\rho)^{n-1} x^j
=
\frac{nP_{n-1,\rho}(x)}{(1-x)^{n}}
\end{equation}
and thus (for all $x$, since we deal with polynomials)
\begin{equation}\label{fys}
  \frac{\partial}{\partial\rho}{P_{n,\rho}(x)}
= n(1-x)P_{n-1,\rho}(x),
\qquad n\ge1.
\end{equation}
Equivalently, by \eqref{efa},
\begin{equation}\label{kem}
    \frac{\partial}{\partial\rho}{A_{n,k,\rho}}
= n\bigpar{A_{n-1,k,\rho}-A_{n-1,k-1,\rho}}, 
\qquad n\ge1.
\end{equation}

\begin{remark}\label{Rroots}
As remarked by  \citet{Frobenius} in the case $\rho=1$,
 it follows from the recursion \eqref{efrec} that if $0<\rho\le1$, the roots
 of $\pnrho$ are real, negative and simple, see \cite{terMorsche}. 
(To see this, use induction and consider the values of $P_{n,\rho}$ at the
 roots of $P_{n-1,\rho}$, and at $0$ and $-\infty$; 
it follows from \eqref{efrec} that these values will be of
   alternating signs, and thus there must be roots of $P_{n,\rho}$ between
   them, and this accounts for all roots of $P_{n,\rho}$.
We omit the details.
The argument also shows that the roots of $P_{n,\rho}$ and $P_{n-1,\rho}$
are interlaced.
For more general results of this kind, see \eg{} \cite{WangYeh} and
\cite[Proposition 3.5]{LiuWang}.) 
Recall that if $0<\rho<1$ there are $n$ roots, and if $\rho=1$ only $n-1$
(for $n\ge1$); in this case we can regard $-\infty$ as an additional root.
Furthermore, by \eqref{pn01} this extends to $\rho=0$: $P_{n,0}$ has 
$n$ roots which are simple, real and non-positive, with 0 being a root in
this case (for $n\ge1$).

Hence, for $0\le\rho\le1$ and $n\ge1$, $\pnrho$ has $n$ roots
$-\infty \le -\gl_{n,n}<\dots<-\gl_{n,1}\le0$. It follows that the
\pgf{} \eqref{pgf} of $\EUrho$ can be written as
\begin{equation}
  \E x^{\EUrho} = \frac{\pnrho(x)}{\pnrho(1)}
=\prod_{j=1}^n\frac{x+\gl_{n,j}}{1+\gl_{n,j}},
\end{equation}
which shows that 
\begin{equation}\label{pb}
  \EUrho\eqd\sumjin I_j
\end{equation}
where $I_j\sim\Be(1/(1+\gl_{n,j}))$ are independent indicator variables.
This stochastic representation can be used to show asymptotic 
properties of the \EFx numbers from standard results for sums of
independent random variables, see 
\cite{Carlitz-asN} ($\rho=1$),
\cite{GawronskiN} 
and \refS{SCLT}.

The fact that all roots of $P_{n,\rho}$ are real implies further that the
sequence $A_{n,k,\rho}$, $k=0,\dots,n$, is log-concave
(\cite[Theorem 53, p.~52]{HLP}; see also Newton's inequality
\cite[Theorem 51, p.~52]{HLP}).
In particular, the sequence is unimodal.
In other words, the distribution of $\EUrho$ is log-concave and unimodal.

Further results on the asymptotics and distribution of the roots of $\pnrho$
are given in  \eg{}
\cite{FaldeyG},
\cite{Gawronski-asymp},
\cite{Reimer:main},
\cite{GawronskiS}.
\end{remark}

\begin{remark}\label{Reuler}
 The Eulerian polynomials and numbers should not be confused with the 
\emph{Euler polynomials} and \emph{Euler numbers}, but there are well-known
connections. (See \eg{} \cite[\S\S 8, 17]{Frobenius};
see also \cite[\S24.1]{NIST} and the references there for some historical
remarks on names and notations.)
First,
  the \emph{Euler polynomials} $E_n(x)$ are defined by their generating
  function \cite[\S24.2]{NIST}
  \begin{equation}\label{ep1}
	\sumn E_n(x)\frac{z^n}{n!} = \frac{2e^{xz}}{e^z+1}.
  \end{equation}
Taking $x=-1$ in \eqref{gf} we see that
\begin{equation}\label{ep2}
  E_n(\rho)=2^{-n}P_{n,\rho}(-1)=2^{-n}\sumkn A_{n,k,\rho}(-1)^k.
\end{equation}
Similarly, the \emph{Euler numbers} $E_n$
\cite[A000364]{OEIS}, 
\cite[\S24.2]{NIST}, 
which are defined as the
coefficients in the Taylor (Maclaurin) series
\begin{equation}
\frac1{\cosh t} 
=\frac{2e^t}{e^{2t}+1}
= \sumn E_n \frac{t^n}{n!}
\end{equation}
are by \eqref{ep1}--\eqref{ep2} given by
\begin{equation}\label{En}
  E_n = 2^n E_n(1/2) = P_{n,1/2}(-1).
\end{equation}
The Euler numbers $E_n$ vanish for odd $n$, and the numbers $E_{2n}$
alternate in sign, $E_{2n}=(-1)^n|E_{2n}|$; the positive numbers $|E_{2n}|$
are also known as \emph{secant numbers} since they are the coefficients in
\cite[p.~432]{E212}
\begin{equation}
\sec t \= \frac1{\cos t} = \sumn |E_{2n}| \frac{t^{2n}}{(2n)!}.  
\end{equation}

Furthermore, returning to the classical case (Euler's case)
$\rho=1$ and taking $x=\ii$ in \eqref{gf} we find
\begin{equation*}
  \begin{split}
\sumn\frac{P_{n,1}(\ii)}{(1-\ii)^{n+1}}\frac{t^n}{n!}
&=
\frac{e^t}{1-\ii e^t}
=  \frac{e^t(1+\ii e^t)}{1+e^{2t}}
=\frac\ii2+ \frac\ii2 \frac{e^{2t}-1}{e^{2t}+1}+ \frac{ e^t}{e^{2t}+1}
\\ &
=\frac\ii2+ \frac\ii2 \tanh t+ \frac 1{2 \cosh t}
\\ &=
\frac{\ii}2
+\frac{\ii}2\summ (-1)^m T_{2m+1} \frac{t^{2m+1}}{(2m+1)!}
+\frac{1}2\summ E_{2m} \frac{t^{2m}}{(2m)!},
  \end{split}
\end{equation*}
where $T_n$ are the \emph{tangent numbers} 
\cite[A000182]{OEIS}, \cite[\S24.15]{NIST}
defined as the coefficients in
the Taylor (Maclaurin) series 
\begin{equation}
\tan t = \sumn T_n \frac{t^n}{n!} = \summ T_{2m +1}\frac{t^{2m+1}}{(2m+1)!}
;
\end{equation}
note that
$T_n=0$ when $n$ is even.
Hence, for $n\ge1$,
\begin{equation}\label{pi}
  P_{n,1}(\ii)=\sum_{k=0}^{n-1}\euler nk \ii^k
=
\begin{cases}
 2^m\ii^m  T_{n}, & n=2m+1, \\
(1-\ii)2^{m-1}(-\ii)^m E_{n},
& n=2m.
\end{cases}
\end{equation}
For even $n$ we can also write this as
\begin{equation}\label{pia}
P_{2m,1}(\ii)
  =(\ii^m-\ii^{m+1})2^{m-1}|E_{2m}|.
\end{equation}
For odd $n$ we can use the relation  \cite[24.15.4]{NIST}
\begin{equation}\label{tb}
  T_{2m-1}=(-1)^{m-1}\frac{2^{2m}(2^{2m}-1)}{2m}B_{2m}
\end{equation}
with
the \emph{Bernoulli numbers} $B_{2m}$
\cite[\S24.2]{NIST} 
and write \eqref{pi} 
as 
\begin{equation}
  P_{n,1}(\ii)=(-2i)^{m}\frac{2^{n+1}(2^{n+1}-1)}{n+1}B_{n+1},
\qquad n=2m+1.
\end{equation}

Similarly, taking $\rho=1$ and $x=-1$ in \eqref{gf}
we find 
\begin{equation*}
  \begin{split}
\sumn\frac{P_{n,1}(-1)}{2^{n+1}}\frac{t^n}{n!}
&=
\frac{e^t}{e^t+1}
=\frac12+ \frac12 \frac{e^{t}-1}{e^{t}+1}
=
\frac12 + \frac12\tanh \frac t2
\\ &=
\frac12+\frac{1}2\summ (-1)^{m}T_{2m+1} \frac{(t/2)^{2m+1}}{(2m+1)!}.
  \end{split}
\end{equation*}
Hence, for $n\ge1$,
\begin{equation}
P_{n,1}(-1) 
= \sum_{k=0}^{n-1}(-1)^k\euler nk 
=   (-1)^{(n-1)/2}T_n
=
\begin{cases}
  (-1)^m T_{n}, & n=2m+1, \\
0, & n=2m,
\end{cases}
\end{equation}
which for odd $n$ can be expressed in $B_{n+1}$ using \eqref{tb}.

Frobenius \cite{Frobenius} studied also $P_{n,1}(\zeta)$ for other roots of
  unity   $\zeta$, using the name 
\emph{Euler numbers of the $m$th order} 
for $P_{n,1}(\zeta)/(\zeta-1)^m$ 
when
  $\zeta$ is a primitive $m$th root of unity (see also
  \cite[p.~163]{Sylvester}); the standard Euler numbers above are the case 
$m=4$ ($\zeta=\ii$) apart from a factor $1+\ii$, see \eqref{pi}--\eqref{pia}.
\end{remark}

\begin{remark}
  For $p\in(0,1)$, \eqref{ef} can be rewritten as
  \begin{equation}
\frac{P_{n,\rho}(p)}{(1-p)^{n}}
=
\sumj (j+\rho)^n (1-p)p^j
=\E (X_p+\rho)^n,	
  \end{equation}
where $X_p\sim\Ge(p)$ has a geometric distribution. In particular,
the moments of a geometric distribution are, using \eqref{pn01},
for $n\ge1$,
\begin{equation}
\E X_p^n = (1-p)^{-n}{P_{n,0}(p)}
= p(1-p)^{-n}{P_{n,1}(p)},
\end{equation}
and the central moments are, using \eqref{psymm},
\begin{equation}
  \begin{split}
\E\xpar{X_p-\E X_p}^n
&=
\E\Bigpar{X_p-\frac{p}{1-p}}^n
=
\frac{P_{n,-p/(1-p)}(p)}{(1-p)^{n}}
\\&
=\Bigparfrac{p}{1-p}^n{P_{n,1/(1-p)}\Bigparfrac1p}.
  \end{split}
\end{equation}
In particular, for $p=1/2$ we have the moments
$\E X_{1/2}^n = 2^{n-1}P_{n,1}(1/2)$ ($n\ge1$)
\cite[A000670]{OEIS} 
(\emph{numbers of preferential arrangements}, 
also called \emph{surjection  numbers}; 
see further \eg{} \cite[Exercise 7.44]{CM} and \cite[II.3.1]{FlajoletS})
and the central moments
$\E (X_{1/2}-1)^n = P_{n,2}(2)$ \cite[A052841]{OEIS}. 
\end{remark}

\begin{remark}
\citet{Benoumhani} studied polynomials related to \EFx
  polynomials. His $F_m(n,x)$ can be expressed as
  \begin{equation}
F_m(n,x)=m^n(1+x)^n P_{n,1/m}\parfrac{x}{1+x}
=\sumkn m^n A_{n,k,1/m} x^k (1+x)^{n-k}.
  \end{equation}
\end{remark}

\section{Splines}\label{Aspline}

As said in \refS{SSn}, the density function $f_n$ is continuous for $n\ge2$,
while $f_1$ has jumps at $0$ and $1$. 
More generally, $f_n$ is $n-2$ times continuously differentiable, 
while $f_n^{(n-1)}$ has jumps at the integer points $0,\dots,n$;
furthermore, $f_n$ is a polynomial of degree $n-1$ in any interval
$(k-1,k)$.
Such functions are called 
\emph{splines} of degree $n-1$, with knots at the integers, see \eg{}
\cite{Schoenberg46, Schoenberg,Schumaker}.
Hence $f_n$ 
is a {spline} of degree $n-1$, with knots at the integers, which
moreover
vanishes outside $[0,n]$; in this context, $f_n$ is known as a
\emph{B-spline}, see \eg{} \cite[Lecture 2]{Schoenberg}.
Here ``B'' stands for basis,
since translates of $f_n$ form a
basis in the linear space $\cS_{n-1}$
of all splines of degree $n-1$ with integer knots
\cite{Schoenberg46,Schoenberg}; 
in other words, every spline $g\in \cS_n$ can be written
\begin{equation}
  \label{bspline}
g(x)=\sumkoooo c_k f_{n+1}(x-k)
\end{equation}
for a unique sequence $(c_k)_{-\infty}^\infty$ of complex numbers
(or real numbers, if we consider real splines),
and conversely, every such sum gives a spline in $\cS_n$. (The sum converges
trivially pointwise.)
The interpretation of the B-spline as the density function of $S_n$ was
observed already by \citet[3.17]{Schoenberg46}.

By \refT{T3}, the values of the B-spline are given by the \EFx
numbers; equivalently, the B-splines satisfy the recursion formula
\eqref{frec}, which is well-known in this setting 
 \cite[(4.52)--(4.53)]{Schumaker}. 

A related construction,
see \eg{} \cite{Schoenberg},
is the \emph{exponential spline},
defined by taking $c_k=t^k$ in \eqref{bspline} for some complex $t\neq0$, \ie,
\begin{equation}\label{expspline}
  \Phi_n(x;t)\=\sumkoooo t^k f_{n+1}(x-k) =\sumkoooo t^{-k} f_{n+1}(x+k),
\end{equation}
which is a spline of degree $n$ satisfying
\begin{equation}\label{sigr}
  \Phi_n(x+1;t)= t \Phi_n(x;t)
\end{equation}
(and, up to a constant factor, the only such spline).
By \eqref{sigr}, $\Phi_n(x;t)$ is determined by its restriction to \oi, 
and
by \eqref{expspline}, \eqref{t3}, \eqref{efa} and \eqref{psymm}
we have for $0\le x\le 1$,
\begin{equation}\label{phip}
  \Phi_n(x;t) = \frac1{n!}\sumkoooo A_{n,k,x} t^{-k}
= \frac1{n!} P_{n,x}(t\qw)
= \frac{t^{-n}}{n!} P_{n,1-x}(t).
\end{equation}
Note that here we rather consider $P_{n,1-x}(t)$ as a polynomial in $x$,
with a parameter $t$, instead of the opposite as we usually do.

These and other relations betweens splines and \EFx polynomials
have been 
known and used for a long time,
see \eg{}
\cite{He,MeinardusMerz,
Merz1, Merz2,
Plonka,
Reimer:extremal,
Reimer:main,
ReimerS,
Schoenberg,
Siepmann-Dr,
Siepmann,
terMorsche,
terMorsche:relations,
WangXuXu}. 
We give a few further examples; see the references
just given for details and further results.

First, consider the following interpolation problem: 
Let $\gl\in\oi$ be given and find a spline
$g\in \cS_n$ such that
\begin{equation}\label{interpol}
  g(k+\gl)=a_k, \qquad k\in\bbZ,
\end{equation}
for a given sequence $(a_k)_{-\infty}^\infty$.
It is not difficult to see that this problem always has a solution, and that
the space of solutions has dimension $n$ if $\gl\in(0,1)$ and $n-1$ if
$\gl\in\setoi$. (If $0<\gl<1$, we may \eg{} choose $c_{-n+1},\dots,c_0$
arbitrarily, and then choose $c_1,c_2,\dots$ and $c_{-n},c_{-n-1},\dots$
recursively so that \eqref{interpol} holds; the case $\gl=0$ or 1 is similar.)
Moreover, the null space, \ie, the space of splines $g\in\cS_n$ such that
$g(k+\gl)=0$ for all integers $k$, contains by \eqref{sigr} every
exponential spline $\Phi_n(x;t)$ such that $\Phi_n(\gl,t)=0$; by \eqref{phip}
this is equivalent to $P_{n,1-\gl}(t)=0$.
Since $\Phi_{n,1-\gl}$ has $n$ non-zero roots if $\gl\in(0,1)$ and $n-1$
if $\gl\in\setoi$, 
see \refR{Rroots}, the exponential splines $\Phi_n(x;t_i)$, where $t_i$ is a
non-zero
root of $P_{n,1-\gl}$, form a basis of the null space of \eqref{interpol}.
Note that these roots $t_i$ are real and negative by \refR{Rroots}.

The cases $\gl=0$ and $\gl=1/2$ are particularly important, which explains
the importance of $P_{n,1}(x)$ and $P_{n,1/2}(x)$ and the corresponding
Eulerian numbers $A_{n,k,1}$ and the Eulerian numbers of type B
$B_{n,k}=2^nA_{n,k,1/2}$ in spline theory.

Similarly, one may consider the periodic interpolation problem, considering
only functions and sequences with a given period $N$. By simple Fourier
analysis, 
if $\go_N=\exp(2\pi\ii /N)$, then
the exponential splines $\Phi(x,\go_N^j)$, $j=1,\dots,N$,
form a basis of the $N$-dimensional 
space of periodic splines in $\cS_n$;
here $\go_N^j=\exp(2\pi\ii j/N)$ ranges over the $N$:th unit roots.
Moreover, 
\eqref{interpol} has a unique
periodic solution for every periodic sequence $(a_k)_{-\infty}^\infty$ if
and only if  
none of these exponential splines vanishes at $\gl$, \ie, if and only if
$P_{n,1-\gl}(\go_N^j)\neq0$ for all $j$. Since the roots of $P_{n,1-\gl}$ 
lie in $(-\infty,0]$, the only possible problem is for $-1$, so this holds
always if $N$ is odd, and if $N$ is even unless $P_{n,1-\gl}(-1)=0$;
by \eqref{ep2} and standard properties of the Euler polynomials
\cite[\S24.12(ii), see also (24.4.26), (24.4.28), (24.4.35)]{NIST},
the
periodic interpolation problem \eqref{interpol} thus has a unique solution
except if either $N$ even, $n$ even and $\gl\in\setoi$, or 
$N$ even, $n$ odd and $\gl=1/2$.

\newcommand\AAP{\emph{Adv. Appl. Probab.} }
\newcommand\JAP{\emph{J. Appl. Probab.} }
\newcommand\JAMS{\emph{J. \AMS} }
\newcommand\MAMS{\emph{Memoirs \AMS} }
\newcommand\PAMS{\emph{Proc. \AMS} }
\newcommand\TAMS{\emph{Trans. \AMS} }
\newcommand\AnnMS{\emph{Ann. Math. Statist.} }
\newcommand\AnnPr{\emph{Ann. Probab.} }
\newcommand\CPC{\emph{Combin. Probab. Comput.} }
\newcommand\JMAA{\emph{J. Math. Anal. Appl.} }
\newcommand\RSA{\emph{Random Struct. Alg.} }
\newcommand\ZW{\emph{Z. Wahrsch. Verw. Gebiete} }
\newcommand\DMTCS{\jour{Discr. Math. Theor. Comput. Sci.} }

\newcommand\AMS{Amer. Math. Soc.}
\newcommand\Springer{Springer-Verlag}
\newcommand\Wiley{Wiley}

\newcommand\vol{\textbf}
\newcommand\Vol{\vol} 
\newcommand\jour{\emph}
\newcommand\book{\emph}
\newcommand\inbook{\emph}
\def\no#1#2,{\unskip#2, no. #1,} 
\newcommand\toappear{\unskip, to appear}

\newcommand\urlsvante{\url{http://www.math.uu.se/~svante/papers/}}
\newcommand\arxiv[1]{\url{arXiv:#1.}}
\newcommand\arXiv{\arxiv}

\end{document}